\documentclass[a4paper,10pt]{amsart}
\usepackage{amssymb,bm,color,bbm,mathrsfs}
\usepackage{hyperref}
\usepackage[arrow,matrix]{xy}
\allowdisplaybreaks

\newcommand {\Hom}{\operatorname {Hom}}

\newcommand {\im}{\operatorname {Im}}

\newcommand {\id}{\operatorname {id}}
\newcommand {\der}{\operatorname {Der}}
\newcommand {\sgn}{\operatorname {sgn}}

\newcommand {\PH}{\operatorname {HP}}
\newcommand {\vol}{\operatorname {vol}}
\newcommand {\de}{\operatorname {d}}
\newcommand {\Der}{\operatorname {Der}}

\begin{document}
\setlength{\baselineskip}{1.4em}
\numberwithin{equation}{section}
\newtheorem{thm}{Theorem}[section]
\newtheorem{prop}[thm]{Proposition}
\newtheorem{lem}[thm]{Lemma}
\newtheorem{cor}[thm]{Corollary}
\theoremstyle{definition}
\newtheorem{defn}[thm]{Definition}
\newtheorem{rmk}[thm]{Remark}
\newtheorem{exam}[thm]{Example}

\title[Poincar\'{e} duality and BV structure]
{Poincar\'{e} duality for
smooth Poisson algebras and BV structure on Poisson Cohomology}
\author{J.~Luo}
\address{Mathematics and Science College, Shanghai Normal University, Shanghai 200234, China}
\email{luojuan@shnu.edu.cn}
\author{S.-Q.~Wang}
\address{School of Mathematics, East China University of Science and Technology, Shanghai 200237, China}
\email{sqwang@ecust.edu.cn}
\author{Q.-S.~Wu}
\address{School of Mathematical Sciences, Fudan University, Shanghai 200433, China}
\email{qswu@fudan.edu.cn}

\begin{abstract} Similar to the modular vector fields in Poisson geometry, modular derivations are defined for smooth Poisson
algebras with trivial canonical bundle. By twisting Poisson module with the modular derivation, the Poisson cochain complex with values in any Poisson module is proved to be isomorphic
to the Poisson chain complex with values in the corresponding twisted Poisson module. Then a version of twisted Poincar\'{e} duality is proved between the Poisson homologies and  cohomologies. Furthermore, a notion of pseudo-unimodular Poisson structure is defined. It is proved  that the Poisson cohomology as a Gerstenhaber algebra admits a Batalin-Vilkovisky operator inherited from some one of its Poisson cochain complex if and only if the Poisson
structure is pseudo-unimodular.
This generalizes the geometric version due to P. Xu.
The modular derivation and Batalin-Vilkovisky operator are also described by using the dual basis of the K\"{a}hler differential module.
\end{abstract}



\maketitle

\setcounter{section}{-1}
\section{Introduction}

Poisson algebras and their (co)homology theory play an important role in the study of their deformation quantization algebras, and vice-versa. 
For example,
the Hochschild homology and cyclic homology of some noncommutative algebras, such as  3-dimensional graded Calabi-Yau algebras \cite{VdB, Mar, BP} and 4-dimensional Sklyanin algebra \cite{TaP1,TaP2}, have been calculated by using deformation theory of Poisson algebras and  Brylinski spectral sequence \cite{Bry}.
Dolgushev proved that the Van den Bergh duality holds for the deformation quantization of unimodular Poisson algebras \cite{Do}.  The third author of the current paper and Zhu considered the filtered deformations of Poisson algebras and proved that the deformation algebra  is Calabi-Yau if and only if the Poisson algebra  is unimodular under some mild assumptions \cite{WZ}. Chen-Chen-Eshmatov-Yang studied the Poincar\'{e} duality between the Poisson homology and cohomology of polynomial algebras with unimodular quadratic Poisson structures, and showed that Kontsevich's deformation quantization as well as Koszul duality preserve the corresponding Poincar\'{e} duality \cite{CCEY21}.

Going back to  the Poisson algebras side,
Launois-Richard  \cite{LR} and Zhu  \cite{Zhu}  proved a twisted Poincar\'{e} duality for polynomial Poisson algebras with quadratic and linear Poisson structures, respectively. In  \cite{LWW}, we generalized their results of twisted Poincar\'{e} duality to any polynomial Poisson algebra. The duality is achieved by twisting the Poisson module structure in a canonical way, which is constructed from the modular derivation \cite[Theorem 3.5]{LWW}.  Later,  L\"{u}-Wang-Zhuang recovered the duality for any affine Poisson algebra with free  K\"{a}hler differential module via its Poisson enveloping algebra \cite{LWZ20}.  Note that the twisted Poincar\'{e} duality reduces to the classical Poincar\'{e} duality  if the Poisson structure is unimodular  \cite{Xu,LWW}. 

In this paper, for any smooth Poisson algebra with trivial canonical bundle, we describe the modular derivations by using the dual basis of  its K\"{a}hler differential module (see Theorem \ref{modular der of smooth}),
and prove that a twisted Poincar\'{e} duality always holds for such kind of Poisson algebras (see Theorem \ref{main-theorem1}).  Comparing with Huebschmann's work \cite{Hue1}, where a  general duality theorem is proved in the setting of Lie-Rinehart algebras,  the duality established in this paper  is constructed from an explicit isomorphism between the Poisson cochain complex of the Poisson algebra $R$ with values in a Poisson module $M$ and
the Poisson chain complex of $R$ with values in the twisted  Poisson module $M_t$ (see Theorem \ref{comm diag1}).


Batalin-Vilkovisky (BV for short) algebras appear in the research of BV formalism, which is a universal and effective method in quantization of gauge field and plays an important role in quantum field theory and string theory \cite{Mne19}. The BV structures on the Hochschild (co)homology of noncommutative algebras have been considered by many researchers, such as  \cite{Gin, Tra, KoKr, LZZ}. For a Poisson algebra, the BV structure on the Poisson (co)homology also attracts many attentions, since the Poisson (co)homology is closely related to the Hochschild (co)homology of its deformation quantization algebra. For any smooth Poisson algebra with trivial canonical bundle, we prove that there is a natural BV algebra structure on its Poisson cochain complex (see Theorem  \ref{main thm-BV}), with the BV operator induced from its de Rham differential on the Poisson chain complex by using the twisted Poincar\'{e} duality given in Theorem \ref{comm diag1}. The BV operator on the Poisson cochain complex is described in Theorem \ref{Description-BV-Ooerator}.  If the Poisson structure is unimodular, then the Poisson cohomology has a BV algebra structure with the BV operator induced from the one on the Poisson cochain complex  (see Theorem \ref{BV for unim poly}).

In the last part of this paper we define a class of pseudo-unimodular  Poisson structures (see Definition~\ref{def of pseudo-unimodular}), which can be  viewed as a generalization of unimodular Poisson structures.
%
For such kind of Poisson structure, we prove that its Poisson cohomology still has a BV algebra structure. Similar results are proved in \cite{LWW1} for Frobenius Poisson algebras.  In fact, we prove that the Poisson cohomology admits a BV operator inherited from some one of its Poisson cochain complex if and only if the Poisson structure is pseudo-unimodular (see Theorem \ref{Pseudo-unimodular has BV} and Corollary \ref{main-cor}).

This paper is organized as follows.
In Section 1, we collect some facts about multi-derivations and higher differential forms on smooth algebras.
In Section 2, after recalling the definitions of Poisson (co)homology,
we prove Theorem \ref{modular der of smooth}, which describes the modular derivation for any smooth Poisson algebra with trivial canonical bundle.
In Section 3, we prove Theorems \ref{comm diag1} and \ref{main-theorem1},
which establish the twisted Poincar\'{e} duality  between the Poisson homologies and cohomologies for smooth Poisson algebras with trivial canonical bundle.
In Section 4, we study the BV algebra structure on the Poisson cochain complex 
and Poisson cohomology, and prove Theorems \ref{main thm-BV} and
\ref{BV for unim poly}.
In the last section, we introduce a notion of pseudo-unimodular Poisson algebras, and prove Theorem \ref{Pseudo-unimodular has BV} and Corollary \ref{main-cor}, which say that  the Poisson cohomology admits a BV operator inherited from its Poisson cochain complex if and only if the Poisson algebra is pseudo-unimodular.

\section{Preliminaries}\label{sec:def.prelimi}
In this section, we collect some necessary facts about multi-derivations, higher differential forms, and contraction maps.
Let $\mathbbm{k}$ be a field. All vector spaces and  algebras 
are over $\mathbbm{k}$. We refer to \cite{LPV} as the basic reference. 
\subsection{Derivations and K\"{a}hler differentials}
Let $R$ be a commutative algebra  and  $M$ be an $R$-module. Let $\Omega^1(R)$ be the module of K\"{a}hler differentials and
$\der(R,M)$ be the set of $\mathbbm{k}$-linear derivations from $R$
to $M$.
There is a canonical isomorphism of left $R$-modules
\begin{equation}\label{der-diff dual}
\Hom_R(\Omega^1(R), M) \to \der(R, M), f \mapsto f \de
\end{equation}with the inverse map $\xi \mapsto f_{\xi}$
where $f_{\xi}: \Omega^1(R) \to  M$ is the map $a\de\!b \mapsto
a\xi(b)$.  Sometimes,  $\der(R, M)$ is
identified with $\Hom_R(\Omega^1(R), M)$ by the isomorphism in \eqref{der-diff dual}. So,
when $\xi \in \der(R, M)$ is viewed as an element in
$\Hom_R(\Omega^1(R), M)$, $\xi(\de\!b)=\xi(b)$; and when $f \in
\Hom_R(\Omega^1(R), M)$ is viewed as an element in $\der(R, M)$,
$f(b)=f(\de\!b)$.
Set $\der(R)=\der(R,R)$.
\subsection{Multi-derivations and higher differential forms}
Let $\mathfrak{X}^p(M)$ be the set of all skew-symmetric $p$-fold
$\mathbbm{k}$-linear multi-derivations with values in $M$, that is,
$$\mathfrak{X}^p(M)=\{F \in \Hom_{\mathbbm{k}}(\wedge^p R, M) \mid F \, \textrm{is a derivation in each argument}\}.$$
Obviously, $\mathfrak{X}^0(M)=M$ and $\mathfrak{X}^1(M)=\der(R, M)$.  An element in $\mathfrak{X}^p(M)$ is called a $p$-fold {\it multi-derivation} form $R$ to $M$.
Set $\mathfrak{X}^*(M)=\oplus_{p\in \mathbb{N}}\mathfrak{X}^p(M)$.

Let $\Omega^p(R)=\wedge_R^p\Omega^1(R)$  be the $p$-th wedge product of the $R$-module
$\Omega^1(R)$ for $p\in \mathbb{N}$,  and $\Omega^{*}(R)=\oplus_{p\in \mathbb{N}}\Omega^{p}(R)$.
An element in $\Omega^p(R)$ is called a {\it K\"{a}hler $p$-form} of $R$. Then the following more general fact holds.

\begin{lem}\label{Xp=Hom-p}
Let $R$ be a commutative algebra, $M$ be an $R$-module. For any $p
\in \mathbb{N}$,
 \begin{equation}\label{chi-diff}
 \mathfrak{X}^p(M) \cong  \Hom_R(\Omega^p(R), M).
\end{equation}
\end{lem}
\begin{proof}
  Let $\varphi: \mathfrak{X}^p(M) \to  \Hom_R(\Omega^p(R), M)$ be the map $F \mapsto \varphi(F): \Omega^p(R) \to M$
  such that,  for any $a_0\de\!a_1\wedge \de\!a_2 \wedge \cdots \wedge \de\!a_p \in \Omega^p(R),$
  $$\varphi(F)(a_0\de\!a_1\wedge \de\!a_2 \wedge \cdots \wedge \de\!a_p)= a_0F(a_1 \wedge a_2 \wedge \cdots \wedge a_p).$$

  Let $\psi:  \Hom_R(\Omega^p(R), M) \to \mathfrak{X}^p(M)$ be the map $g \mapsto \psi(g)$ such that, for any
  $a_1 \wedge a_2 \wedge \cdots \wedge a_p \in \wedge^p R$,
  $$ \psi(g)(a_1 \wedge a_2 \wedge \cdots \wedge a_p)=g(\de\!a_1 \wedge \de\!a_2 \wedge \cdots \wedge \de\!a_p).$$

Then $\varphi (\psi(g)) = g$ and $\psi (\varphi(F))=F$.
It follows that $\varphi$ is an isomorphism with the inverse $\psi$.
\end{proof}

\begin{defn}\label{F-wedge-G}
For any $F\in \mathfrak{X}^{p}(R)$ and $G\in \mathfrak{X}^{q}(M)$,  define a product
$F\wedge G \in \mathfrak{X}^{p+q}(M)$ as : for any $a_1, a_2, \cdots,  a_{p+q} \in R$,
\begin{align*}
&(F\wedge G)(a_1\wedge a_2\wedge \cdots \wedge a_{p+q})\\
= &\sum\limits_{\sigma\in S_{p,q}}
 \sgn(\sigma) F(a_{\sigma(1)}\wedge a_{\sigma(2)} \wedge \cdots \wedge a_{\sigma(p)}) G(a_{\sigma(p+1)}\wedge a_{\sigma(p+2)} \wedge \cdots \wedge a_{\sigma(p+q)}),
\end{align*}
 where $S_{p, q}$ denotes the set of all $(p,q)$-shuffles, which
 are the permutations $\sigma \in S_{p+q}$ such that $\sigma(1)< \cdots <
 \sigma(p)$ and $\sigma(p+1)< \cdots < \sigma(p+q)$.
\end{defn}
It is easy to check that:
\begin{prop} $(\mathfrak{X}^{*}(R), \wedge)$ is a graded commutative  $R$-algebra. $\mathfrak{X}^{*}(M)$ is a graded $\mathfrak{X}^{*}(R)$-module.
\end{prop}

Let $\wedge^p_R\der(R)$ be the $p$-th wedge product of $\der(R)$.
There is a natural map 
\begin{equation}\label{map-alpha}
    \alpha: \wedge^p_R \der(R) \to
\mathfrak{X}^p(R), \, \xi_1 \wedge \xi_2 \wedge \cdots \wedge \xi_p
\mapsto\alpha(\xi_1 \wedge \xi_2\wedge \cdots \wedge \xi_p),
\end{equation} where $\alpha(\xi_1 \wedge \xi_2\wedge \cdots \wedge \xi_p)$
is the map $\wedge^p R \to R$,
$$a_1 \wedge a_2\wedge \cdots \wedge a_p \mapsto
\sum_{\sigma \in S_p} \sgn(\sigma) \xi_1(a_{\sigma(1)})
\xi_2(a_{\sigma(2)}) \cdots \xi_p(a_{\sigma(p)})
=\begin{vmatrix}
\xi_1(a_{1})&\cdots&\xi_1(a_{p})\\
 \vdots &  \ddots & \vdots  \\
\xi_p(a_{1})&\cdots&\xi_p(a_{p})\\
\end{vmatrix}.$$ It is easy to
check that $\alpha$ is well-defined. If $R$ is smooth affine, then $\alpha$
is an isomorphism as proved in the next subsection (see Corollary \ref{smooth duality}), which induces an algebra  isomorphism from the exterior algebra $E_R(\Der(R))$ to $(\mathfrak{X}^{*}(R), \wedge)$.

\subsection{Smooth affine algebras and multi-derivations}
Recall that an affine commutative algebra $R$ is {\it smooth} (over $\mathbbm{k}$) if it satisfies that: for any surjective morphism $\varepsilon: E \to A$  between commutative algebras $E$ and $A$ with $(\ker \varepsilon)^2 = 0$, and any morphism $f: R \to A$, there is a morphism $g: R \to E$ such that $\varepsilon g=f$. 
In fact, $R$ is smooth if and only if its global dimension is finite; if and only if  the projective dimension of $R$ as an $R$-bimodule
 is finite. If $R$ is a smooth affine algebra, then $\Omega^1(R)$ is a finitely generated projective $R$-module  \cite[9.3]{Weibel}.

When $\Omega^1(R)$ is finitely generated  $R$-projective, there
is a canonical isomorphism $$\Omega^1(R)\cong \Hom_R(\der(R), R)$$
following \eqref{der-diff dual} and the dual basis lemma for projective modules.
In fact, $\Omega^p(R)\cong \Hom_R(\wedge^p_R\der(R), R)$ holds for
any $p \in \mathbb{N}$ by the following lemma.
\begin{lem}\label{proj-wedge}
Let $R$ be a commutative algebra, $P$ be a finitely generated
projective $R$-module. Then, for any $p \in \mathbb{N}$, $\wedge^p_R
P$ is projective, and
\begin{equation} \label{proj-wedge-iso}
\wedge^p_R \Hom_R(P, R) \cong  \Hom_R(\wedge^p_R P, R).
\end{equation}
\end{lem}
\begin{proof} Define $\alpha' : \wedge^p_R \Hom_R(P, R) \to
 \Hom_R(\wedge^p_R P, R), f_1 \wedge f_2\wedge \cdots \wedge f_p
\mapsto \alpha'(f_1 \wedge f_2\wedge \cdots \wedge f_p)$, which is
the map $\wedge^p_R P \to R$,
$$y_1 \wedge y_2\wedge \cdots \wedge y_p \mapsto \sum_{\sigma \in S_p} \sgn(\sigma) f_1(y_{\sigma(1)}) f_2(y_{\sigma(2)}) \cdots
f_p(y_{\sigma(p)})=\begin{vmatrix}
f_1(y_{1})&\cdots&f_1(y_{p})\\
 \vdots &  \ddots & \vdots  \\
f_p(y_{1})&\cdots&f_p(y_{p})\\
\end{vmatrix}.$$ It is easy to check that $\alpha'$ is
well-defined.

Let $\{x_i,x_i^*\}_{i=1}^r$ be a dual basis for the projective $R$-module $P$.  
Then $$\{x_{i_1} \wedge x_{i_2} \wedge \cdots \wedge x_{i_p},
\alpha'( x_{i_1}^* \wedge x_{i_2}^* \wedge \cdots \wedge
x_{i_p}^*)\}_{1 \leq i_1<i_2<\cdots<i_p \leq r}$$ is a dual basis
for the projective module $\wedge^p_R P$.

Define $\beta':  \Hom_R(\wedge^p_R P, R) \to \wedge^p_R \Hom_R(P,
R), $
$$f \mapsto \sum\limits_{1 \leq i_1<i_2<\cdots<i_p \leq r}f(x_{i_1}\wedge x_{i_2}
\wedge \cdots \wedge x_{i_p})x_{i_1}^* \wedge x_{i_2}^* \wedge \cdots \wedge
x_{i_p}^*.$$

Then $\alpha'$ and $\beta'$ are inverse to each other.
%
%
%
%
\end{proof}
\begin{cor}\label{smooth duality}
Let $R$ be a smooth algebra. Then the map $\alpha$ defined in \eqref{map-alpha} gives an isomorphism   $\wedge^p_R\der(R) \cong
\mathfrak{X}^p(R)$ for any $p\in \mathbb{N}$. 
Moreover,  the following diagram commutes:
$$
\xymatrix@=1.5cm{
   \wedge^p_R\der(R) \ar[d]_{\alpha}\ar[r]^(0.4){by\; \eqref{der-diff dual}}_(0.4){\cong}
        &   \wedge^p_R \Hom_R(\Omega^1(R), R)   \ar[d]^{by\; \eqref{proj-wedge-iso}}_{\cong}  \\
   \mathfrak{X}^p(R)     \ar[r]^(0.4){by\; \eqref{chi-diff}}_(0.4){\cong}
        & \Hom_R(\Omega^p(R), R).
     }
$$
%
\end{cor}
\begin{proof} Since $R$ is smooth, $\Omega^1(R)$ is a finitely
generated projective $R$-module
.

Suppose $\{(\de\!x_i), (\de\!x_i)^{*}\}_{i=1}^r$ is a dual basis for the projective module
$\Omega^1(R)$. Then the inverse of $\alpha$ is $\mathfrak{X}^p(R)
\to \wedge^p_R\der(R)$, which sends $F$ to

$\sum\limits_{1 \leq
i_1<i_2<\cdots<i_p \leq r} F(x_{i_1}\wedge x_{i_2}
\cdots \wedge x_{i_p})(\de\!x_{i_1})^* \wedge (\de\!x_{i_2})^*
\wedge \cdots \wedge (\de\!x_{i_p})^*$.
\end{proof}

\begin{rmk}
 In Corollary \ref{smooth duality}, the condition ``$R$ is  smooth" can be reduced to that ``$R$ is  commutative  with $\Omega^1(R)$ being a finitely
generated projective $R$-module".
\end{rmk}

\begin{defn} 
Suppose $R$ is a smooth affine algebra and $n\in \mathbb{N}$. We say
$R$ is  smooth of dimension $n$ if $n=\sup\{i\in\mathbb{N}\mid\Omega^i(R)\neq 0\}$. In this case, $n$ is called the {\it smooth dimension} of $R$, and $\Omega^n(R)$ is called the {\it canonical bundle} of $R$.
Moreover, if $\Omega^n(R)\cong R$ (as $R$-modules), then we say the canonical bundle of $R$ is trivial.
\end{defn}

\subsection{Contraction maps}
Let $R$ be a commutative algebra and $M$ be a right $R$-module.

\begin{defn}\label{ctr0}
For any $F\in \mathfrak{X}^{p}(M)$,  the {\it contraction map}
$\iota_{F}: \Omega^{q}(R)\rightarrow M\otimes_R \Omega^{q-p}(R)$  is defined as: when $q < p, \iota_{F}=
0$; when $q\geq p$ and $\omega =a_0\de\! a_1\wedge \de\!
a_2\wedge\cdots\wedge\de\! a_q \in \Omega^{q}(R)$,
$$
\iota_{F}(\omega)= \sum\limits_{\sigma\in S_{p,q-p}}
 \sgn(\sigma) F(a_{\sigma(1)}\wedge a_{\sigma(2)} \wedge \cdots \wedge a_{\sigma(p)})a_0\otimes
 \de\! a_{\sigma(p+1)}\wedge\cdots\wedge\de\! a_{\sigma(q)}.$$
\end{defn}

\begin{rmk}\label{ctr1}
If $M=R$, then for any $F\in \mathfrak{X}^{p}(R)$,
$\iota_{F}: \Omega^{*}(R)\rightarrow \Omega^{*}(R)$ is a graded
$R$-linear map of degree $-p$.
For $F=a\in R = \mathfrak{X}^{0}(R)$, the contraction map
should  be understood as $\iota_{F}(\omega) = a \omega$.
\end{rmk}
\begin{rmk}
Note that for any $F\in \mathfrak{X}^{p}(R)$, the contraction map
$\iota_{F}: \Omega^{*}(R)\rightarrow \Omega^{*}(R)$ is an $R$-module morphism. 
So the map $\id_M\otimes_R \iota_{F}: M\otimes_R \Omega^{q}(R)\rightarrow M\otimes_R \Omega^{q-p}(R)$ is well-defined. Sometimes we also denote $\id_M\otimes_R \iota_{F}$ by $\iota_{F}$ and call it the contraction map induced by $F$ : for any $m\otimes \de\! a_1\wedge \de\!
a_2\wedge\cdots\wedge\de\! a_q \in M\otimes_R \Omega^{q}(R)$,
\begin{align*}
&\iota_{F}(m\otimes \de\! a_1\wedge \de\!
a_2\wedge\cdots\wedge\de\! a_q)\\
= &\sum\limits_{\sigma\in S_{p,q-p}}
 \sgn(\sigma)m F(a_{\sigma(1)}\wedge a_{\sigma(2)} \wedge \cdots \wedge a_{\sigma(p)})\otimes
 \de\! a_{\sigma(p+1)}\wedge\cdots\wedge\de\! a_{\sigma(q)}.
\end{align*}
\end{rmk}

\begin{prop}\label{ctr com}
For any $F\in \mathfrak{X}^{p_1}(M)$ and  $G\in \mathfrak{X}^{p_2}(R)$,
$$\iota_F  \iota_G=(-1)^{p_1p_2}\iota_G  \iota_F : \Omega^{q}(R)\rightarrow M\otimes_R \Omega^{q-p_1-p_2}(R).$$
\end{prop}
\begin{proof}It is easy to check that
$$\iota_F  \iota_G=\iota_{G\wedge F}=(-1)^{pq}\iota_{F\wedge G}=(-1)^{pq}\iota_G  \iota_F.$$
\end{proof}
\begin{defn}\label{ctr2}
For any $\omega \in \Omega^{p}(R)$, the {\it contraction map} $\iota_{\omega}:
\mathfrak{X}^{*}(R)\rightarrow \mathfrak{X}^{*}(R)$ is a graded
$R$-linear map of degree $-p$, which is defined as $\iota_{\omega}:
\mathfrak{X}^{q}(R)\rightarrow \mathfrak{X}^{q-p}(R)$: when $q < p,
\iota_{\omega}= 0$; when $q\geq p$ and $F \in \mathfrak{X}^{q}(R)$,
$$(\iota_{\omega} F) ( a_1 \wedge a_2 \wedge \cdots \wedge a_{q-p})
=F(\de\! a_1\wedge \de\! a_2\wedge\cdots\wedge\de\! a_{q-p} \wedge \omega
),$$ where the action is viewed by identifying $\mathfrak{X}^{q}(R)$ with
$\Hom_R(\Omega^q(R), R)$ via \eqref{chi-diff}.
\end{defn}

\begin{prop} \cite[Proposition 3.4(3)]{LPV}\label{ctr-equality} 
Let $R$ be a commutative algebra and $F\in \mathfrak{X}^{p}(R)$. For
any $a \in R$ and $\omega\in \Omega^q(R)$,
\begin{equation} \label{ctr-equality-no}
\iota_F(\omega \wedge \de\! a)= \iota_{F}(\omega)\wedge \de\! a +
 (-1)^{q-p+1}\iota_{\iota_{\de\!a}(F)}(\omega).
\end{equation}

\end{prop}

\subsection{Contraction maps in smooth case}

Let $R$ be a smooth algebra (or a commutative algebra with $\Omega^1(R)$ being a finitely
generated projective $R$-module). Then, $\wedge^p_R\der(R)\cong
\mathfrak{X}^{p}(R)$ by Corollary \ref{smooth duality}, and
$\Omega^{p}(R)$ can be viewed as the dual module of
$\wedge^{p}_R\der(R)$, via
$$\Omega^p(R) \cong \wedge^p_R\Hom_R(\der(R), R) \cong \Hom_R(\wedge^p_R\der(R), R)\cong\Hom_R(\mathfrak{X}^p(R), R).$$

\begin{defn}\label{ctr3} Let $R$ be a smooth algebra. Then
for any multi-derivation $F\in \mathfrak{X}^{p}(R)$, one can define a
natural contraction operator
$$\iota_{F}: \Omega^{q}(R)\rightarrow \Omega^{q-p}(R),\, \omega \mapsto \iota_{F}(\omega),$$
with $\iota_{F}(\omega)$ given by $$\xi_{p+1} \wedge\xi_{p+2}
\wedge \cdots \wedge \xi_{q} \mapsto \omega(F\wedge\xi_{p+1} \wedge
\cdots \wedge \xi_{q})$$ for any $\xi_{p+1} \wedge \xi_{p+2} \wedge
\cdots \wedge\xi_{q} \in \mathfrak{X}^{q-p}(R)$, i.e., if $F=\xi_1\wedge\xi_2\wedge\cdots\wedge\xi_p,$
$$\iota_{F}(\omega)(\xi_{p+1} \wedge\xi_{p+2}
\wedge \cdots \wedge \xi_{q} )=\omega(\xi_1\wedge\cdots\wedge\xi_p\wedge\xi_{p+1} \wedge \cdots
\wedge \xi_{q}).$$
\end{defn}

\begin{prop}\label{ctr1=3}
The contraction map $\iota_F$ in Remark \ref{ctr1} is the same as in Definition \ref{ctr3} under the canonical isomorphism
$\Omega^q(R)\cong \Hom_{R}(\wedge^{q}_R\Der(R), R)$.
\end{prop}
\begin{proof}
Let $\omega=a_0\de\! a_1\wedge \de\!
a_2\wedge\cdots\wedge\de\! a_q \in \Omega^{q}(R)$, $q\geq p$; $F=\xi_1\wedge\xi_2\wedge\cdots\wedge\xi_p \in \mathfrak{X}^{p}(R)$ and $ \xi_{p+1} \wedge \xi_{p+2} \wedge \cdots \wedge\xi_{q} \in \mathfrak{X}^{q-p}(R)$. Then there are two ways to compute $\iota_{F}(\omega)$ according to
Definitions \ref{ctr0} and \ref{ctr3}, respectively.
The conclusion follows from the Laplace
expansion of the determinants.
\end{proof}
%
\begin{defn}\label{ctr4} Let $R$ be a commutative algebra with $\Omega^1(R)$ being finitely generated projective. Then
for any $\omega \in \Omega^{p}(R)$, one defines a natural contraction
operator  $\iota_{\omega}: \mathfrak{X}^{q}(R) \rightarrow
\mathfrak{X}^{q-p}(R)$, $\xi_1\wedge\xi_2\wedge\cdots\wedge\xi_q \mapsto $
$$
\sum\limits_{\sigma\in S_{q-p, p}} \sgn(\sigma) [(\xi_{\sigma(q-p+1)}\wedge \cdots \wedge \xi_{\sigma(q)})(\omega)]\, \xi_{\sigma
(1)} \wedge \cdots\wedge\xi_{\sigma (q-p)}.$$
\end{defn}

\begin{prop}The contraction map $\iota_{\omega}$ in Definition \ref{ctr2} is the same as in Definition
 \ref{ctr4}  under the canonical isomorphism
 $\mathfrak{X}^{q}(R) \cong \Hom_R(\Omega^q(R), R)$.
\end{prop}
\begin{proof}
The proof is similar to that of Proposition \ref{ctr1=3}.
\end{proof}

\section {Modular derivations of Smooth Poisson algebras}
In this section, we recall some materials on the (co)homology theory of Poisson algebras, and the definition of the modular derivation for smooth Poisson algebras with trivial canonical bundle. In the final part, we give a description of the modular derivation by using the dual basis of the  K\"{a}hler differential module, which is a finitely generated projective module. 
\subsection{Poisson algebras and Poisson modules}
\begin{defn} \cite{Lic, Wei0} A commutative $\mathbbm{k}$-algebra $R$ equipped with a
bilinear map $\{-, -\} : R \times R \to R$ is called a {\it Poisson
algebra} if
\begin{enumerate}
\item 
    $(R, \{-, -\})$ is a $\mathbbm{k}$-Lie algebra;
\item $\{-, -\} : R \times R \to R$  is a derivation in each argument with respect to the multiplication of $R$.
\end{enumerate}
We call such a bilinear map $\pi=\{-, -\} \in \mathfrak{X}^2(R)$ a Poisson structure over $R$.
\end{defn}
\begin{defn}\cite{Oh} A right {\it Poisson module} $M$ over Poisson algebra $R$ is a $\mathbbm{k}$-vector
space $M$ endowed with two bilinear maps $\cdot$ and $\{-, -\}_M : M
\times R \to M $ such that
 \begin{enumerate}
\item $(M,  \cdot)$ is a right module over the commutative algebra $R$;
\item $(M,  \{-, -\}_M)$ is a right Lie-module over the Lie algebra $(R, \{-, -\})$;
\item $\{xa, b\}_M = \{x, b\}_M a + x\{a, b\}$ for any $a, b \in R$ and $x \in M$;
\item $\{x, ab\}_M = \{x, a\}_M  b + \{x, b\}_M a$ for any $a, b \in R$ and $x \in M$.
\end{enumerate}
\end{defn}

Left Poisson modules are defined similarly. Any Poisson
algebra $R$ is naturally a right and left Poisson module over itself.

\subsection{Poisson homology and cohomology}
Let $M$ be a right Poisson module over the  Poisson algebra $R$.
There is a  canonical chain complex
\begin{equation}
 \label{poison-chain-complex}
\cdots \longrightarrow M \otimes_R
\Omega^{p}(R)\stackrel{\partial_p}{\longrightarrow} M \otimes_R
\Omega^{p-1}(R)\stackrel{\partial_{p-1}}{\longrightarrow}\cdots
\stackrel{\partial_2}{\longrightarrow}M \otimes_R \Omega^{1}(R)
\stackrel{\partial_1}{\longrightarrow} M  \to 0
\end{equation}
where $\partial_p \colon M \otimes_R \Omega^{p}(R)\longrightarrow M
\otimes_R \Omega^{p-1}(R)$ is defined as:
\begin{multline*}
 \partial_p(m \otimes \de\!a_1\wedge \cdots \wedge \de\!a_p)=\sum_{i=1}^p(-1)^{i-1}\{m, a_i\}_M\otimes
  \de\!a_1\wedge \cdots \widehat{\de\!a_i} \cdots \wedge \de\!a_p \\
  {}+\sum_{1\le i<j\le p}(-1)^{i+j}m \otimes \de \{a_i, a_j\}
  \wedge \de\!a_1\wedge \cdots  \widehat{\de\!a_i} \cdots  \widehat{\de\!a_j} \cdots \wedge
  \de\!a_p
\end{multline*}
(where $\widehat{\de\!a_i}$ means that $\de\!a_i$ is deleted).
\begin{defn} \cite{Mas}
The complex \eqref{poison-chain-complex} is called the {\it Poisson chain
complex} of $R$ with values in $M$, and its $p$-th homology is
 called the $p$-th {\it Poisson homology} of $R$ with values in
$M$, denoted by $\PH_p(R, M)$.
\end{defn}

In the case $M=R$, $\partial=[\iota_{\pi}, \de]$ where $\de$ is the de Rham differential, and the Poisson homology is the canonical homology given by Brylinski \cite{Bry}.


 There is also a canonical cochain complex
\begin{equation}
 \label{poison-cochain-complex}
0 \longrightarrow M \stackrel{\delta^0}{\longrightarrow} \mathfrak{X}^1(M)
\stackrel{\delta^1}{\longrightarrow} \cdots
\stackrel{\delta^{p-1}}{\longrightarrow} \mathfrak{X}^p(M)
\stackrel{\delta^p}{\longrightarrow} \mathfrak{X}^{p+1}(M) \longrightarrow
\cdots
\end{equation}
where $\delta^p \colon \mathfrak{X}^p(M)\longrightarrow \mathfrak{X}^{p+1}(M)$ is
defined as $F \mapsto \delta^p(F)$ with
\begin{multline*}
  \delta^p(F)(a_1\wedge \cdots \wedge a_{p+1})
  =\sum_{i=1}^{p+1}(-1)^{i}\{F(a_1 \wedge \cdots \widehat{a_i}\cdots \wedge a_{p+1}), a_i\}_M \\
  {}+\sum_{1\le i<j\le p+1}(-1)^{i+j}
  F(\{a_i, a_j\}\wedge a_1\wedge \cdots  \widehat{a_i} \cdots  \widehat{a_j} \cdots \wedge
  a_{p+1}).
\end{multline*}

\begin{defn} \cite{Lic, Hue}
The complex \eqref{poison-cochain-complex} is called the {\it
Poisson cochain complex} of $R$ with values in $M$, and its $p$-th
cohomology is called the $p$-th {\it Poisson cohomology} of $R$ with
values in $M$, denoted by $\PH^p(R, M)$.
\end{defn}

The elements in
$\ker \delta^1$ are called {\it Poisson derivations}, and the
elements in $\im \delta^0$  are called {\it Hamiltonian
derivations}, which are of the form $\{m, -\}_M$ for $m\in M$, denoted by $H_m$.

\begin{exam}
$\PH^0(R, M)=\{m \in M \mid \{m, a\}_M=0,
\forall\, a \in R\}$ is the set of Casimir elements in $M$;
$\PH^1(R, M)=\mbox{\{Poisson derivations\}}/\mbox{\{Hamiltonian
derivations\}}$.
\end{exam}


\subsection{Modular derivations and  Modular class}

\begin{defn}\label{modular4} Let $R$ be a smooth Poisson algebra of
dimension $n$ with trivial canonical bundle $\Omega^n(R)=R \vol$
where $\vol$ is a volume form. The {\it modular derivation} of $R$
with respect to $\vol$ is defined as the map $\phi_{\vol}: R \to R$
such that for any $a\in R$,
$$\phi_{\vol}(a)=\frac{\mathscr{L}_{H_a}(\vol)}{\vol},$$ where  $H_a=\{a,-\}: R \to R$ is the Hamiltonian derivation associated to $a$ and
$\mathscr{L}_{H_a}=[\de, \iota_{H_a}]$ is the Lie derivation. 
\end{defn}

In fact, the modular derivation $\phi_{\vol}$ is not only a derivation, but also a Poisson derivation.
When the volume form is changed, e.g., $\vol' = u \vol$ for some unit $u \in R$, then the corresponding Poisson
derivation $\phi_{\vol'} = \phi_{\vol} - u^{-1}H_u$, which is modified by a so called {\it log-Hamiltonian derivation} $u^{-1}\{-,u\}$ (see \cite{Do}).
The {\it modular class} of $R$ is defined as the class
$\phi_{\vol}$ modulo log-Hamiltonian derivations. If
the modular class is trivial, i.e., $\phi_{\vol}$ is a
log-Hamiltonian derivation, then $R$ is called {\it unimodular}.

\begin{exam}\cite{LWW} Let $R={\mathbbm{k}}[x_1, x_2,
\cdots, x_n]$ be a polynomial Poisson algebra with Poisson bracket
$\{-, -\}.$ Then  $\Omega^1(R)=\oplus_{i=1}^n R \de\! x_i$ and
$\Omega^n(R)=R \de\! x_1 \wedge \de\! x_2 \wedge \cdots \wedge \de\!
x_n$ with $\vol=\de\! x_1 \wedge \de\! x_2 \wedge
\cdots \wedge \de\! x_n$ as a volume form. The modular derivation $\phi_{\vol}$ is given by
 $$\phi_{\vol}(f)=\sum_{j=1}^n \frac{\partial \{f,
x_j\}}{\partial x_j}, \, \forall \,  f \in R.$$
\end{exam}

The following  is an example of smooth algebra with trivial canonical bundle.

\begin{exam}\label{ep-of -smoth}
Let $R=\mathbb{R}[x, y, z]/(x^2 + y^2 + z^2 -1)$. Then $R$ is smooth
of dimension $2$. In fact, it is well known that $$\Omega^1(R) \cong R\de\!x
\oplus R\de\!y \oplus R\de\!z/{R(x\de\!x, y\de\!y, z\de\!z)}$$   is a stably-free and non-free
projective $R$-module, which is generated  by three elements at least
(see \cite[11.2.3 and 15.3.15]{MR}). Furthermore, $$\Omega^2(R)=R(x\de\!y \wedge \de\!z +
y\de\!z \wedge \de\!x + z\de\!x \wedge \de\!y)$$ is a rank $1$ free $R$-module, with a basis element
$\eta=x\de\!y \wedge \de\!z + y\de\!z \wedge \de\!x + z\de\!x \wedge \de\!y$.
Note that $x\cdot \eta=\de\!y\wedge\de\!z, y\cdot \eta=\de\!z\wedge\de\!x$ and $z\cdot
\eta=\de\!x\wedge\de\!y.$

Since
$\de\!x \wedge \de\!y
\wedge \de\!z=(x^2 + y^2 +z^2)(\de\!x \wedge \de\!y \wedge \de\!z)=(x\de\!x+y\de\!y+z\de\!z)\wedge\eta=0,$
$\Omega^3(R)=R(\de\!x \wedge \de\!y \wedge \de\!z)=0.$ So, $R$ is smooth of dimension $2$ with trivial canonical bundle, and $\eta$ is a volume form.

Consider the exact sequence
$$0 \to R(x\de\!x, y\de\!y, z\de\!z) \to R\de\!x \oplus R\de\!y \oplus R\de\!z \to \Omega^1(R) \to 0.$$
It splits with the splitting maps $p: R\de\!x \oplus R\de\!y \oplus R\de\!z \to
R(x\de\!x, y\de\!y, z\de\!z)$:
$$(a\de\!x, b\de\!y, c\de\!z) \mapsto (ax+by+cz)(x\de\!x, y\de\!y, z\de\!z),$$
and $i: \Omega^1(R) \to  R\de\!x \oplus R\de\!y \oplus R\de\!z$:
\begin{align*}
& a\de\!x + b\de\!y + c\de\!z \mapsto (a\de\!x, b\de\!y, c\de\!z)-(ax+by+cz)(x\de\!x, y\de\!y,
 z\de\!z)\\
& = ((a-x(ax+by+cz))\de\!x, (b-y(ax+by+cz))\de\!y, (c-z(ax+by+cz))\de\!z).
\end{align*}

Then there is a dual basis $\{ \de\!x, \de\!y,
\de\!z; (\de\!x)^*,(\de\!y)^*, (\de\!z)^* \}$ for the projective module $\Omega^1(R),$ where
\begin{align*}
&(\de\!x) ^*: \Omega^1(R) \to R, a\de\!x + b\de\!y + c\de\!z \mapsto a-x(ax+by+cz),\\
&(\de\!y) ^*: \Omega^1(R) \to R, a\de\!x + b\de\!y + c\de\!z \mapsto b-y(ax+by+cz),\\
&(\de\!z) ^*: \Omega^1(R) \to R, a\de\!x + b\de\!y + c\de\!z \mapsto c-z(ax+by+cz).
\end{align*}
So we have $(\de\!x) ^*(\de\!x) + (\de\!y) ^*(\de\!y) + (\de\!z) ^*(\de\!z) = 1 -x^2 + 1
-y^2 + 1-z^2 =2.$ It agrees with the conclusion in Lemma \ref{dual-basis=n}.

It follows that $\der(R)$ is generated by $\{(\de\!x)^*,(\de\!y)^*,
(\de\!z)^*\}$. For any $f\in R$,
\begin{align*}
(\de\!x)^*(f)&=(\de\!x)^*(\frac{\partial f}{\partial x} \de\!x +
\frac{\partial f}{\partial y} \de\!y + \frac{\partial f}{\partial z}
\de\!z)
=(1-x^2)\frac{\partial f}{\partial x}  -xy \frac{\partial f}{\partial y} -xz \frac{\partial f}{\partial z},\\
(\de\!y)^*(f)&=(\de\!y)^*(\de\!f)= -xy \frac{\partial f}{\partial x} +(1-y^2) \frac{\partial f}{\partial y} -yz \frac{\partial f}{\partial z},\\
(\de\!z)^*(f)&=(\de\!z)^*(\de\!f)= -xz\frac{\partial f}{\partial x}  - yz \frac{\partial
f}{\partial y} +(1-z^2) \frac{\partial f}{\partial z}.
\end{align*}
\end{exam}
We will come back to this example at the end of this section.
\subsection{Description of the modular derivations}\label{mod-der-in-dual-basis} In this subsection,
we will describe  the modular derivation of a smooth Poisson algebra with trivial canonical bundle in Theorem \ref{modular der of smooth} by using the dual basis for $\Omega^1(R)$.

First, there is  a useful lemma.
\begin{lem} \label{contraction-formula} Let $R$ be a commutative algebra with
$\Omega^{n+1}(R)=0$, and $\eta$ be an $n$-form in $\Omega^n(R)$.
\begin{enumerate}
\item For any $a \in R$, and $F \in \mathfrak{X}^1(R)$, $F(a)\eta=\de\! a\wedge
\iota_{F}(\eta)$.
\item If further, $R$ is a Poisson algebra with Poisson structure
$\pi$, then, for any $a \in R$, $\iota_{H_a}(\eta)=-\de\! a\wedge
\iota_{\pi}(\eta)$.
\end{enumerate}
\end{lem}

\begin{proof} (1) Since $\iota_{\de\! a}(F)=F(a)\in R$, $\iota_{\iota_{\de\! a}(F)}(\eta) =
\iota_{F(a)}(\eta)=F(a)\eta$. On the other hand, $\iota_{F}(\eta
\wedge \de\! a)=0$.  It follows from \eqref{ctr-equality-no} that
$(-1)^{n-1} F(a)\eta = \iota_{F}(\eta)\wedge \de\! a$ and
$F(a)\eta=\de\! a \wedge \iota_{F}(\eta)$.

(2) By taking $F=\pi \in \mathfrak{X}^2(R)$ in equation \eqref{ctr-equality-no}, $\iota_{\de\!
a}(\pi)=\{-,a\}=-H_a\in \mathfrak{X}^1(R)$ and $\iota_{\iota_{\de\!
a}(\pi)}(\eta) = - \iota_{H_a}(\eta)$. It follows from $\iota_{\pi}(
\eta \wedge \de\! a)=0$ and  \eqref{ctr-equality-no} that $\iota_{\pi}(\eta) \wedge \de\! a-(-1)^{n-1}
\iota_{H_a}(\eta)=0$. Hence,
$\iota_{H_a}(\eta)=-\de\! a\wedge \iota_{\pi}(\eta)$.
\end{proof}

In the following, let $R$ be a smooth algebra of dimension $n$ with
trivial canonical bundle and the Poisson structure $\{-, -\}$.  Let
 $\{(\de\!x_i), (\de\!x_i)^{*}\}_{i=1}^r$ be a dual basis for $\Omega^1(R)$,
$\vol \in \Omega^n(R) $ be a volume form and $\phi_{\vol}$ be the
modular derivation of $R$ with respect to $\vol$. Note that the number of generators $r$ may be larger than $n$.

Let $S = \{(I_1,I_2,\cdots ,I_n) \mid I_1,\cdots, I_n \mbox{ are integers and } 1\leq I_1<I_2<\cdots <I_n \leq r\}$. For any
$I=(I_1,I_2,\cdots ,I_n)\in S$, to simplify the notations, let $\de\!x_I$ denote
$\de\!x_{I_1}\wedge \de\!x_{I_2}\wedge \dots \wedge \de\!x_{I_n}$
and $\de\!x_I^*$ denote $(\de\!x_{I_1})^*\wedge
(\de\!x_{I_2})^*\wedge\cdots \wedge (\de\!x_{I_n})^*$.   Then $\{\de\!x_I, \de\!x_I^*\}_{I \in S}$ is
a dual basis for $\Omega^n(R)$ by Lemma \ref{proj-wedge}.
Let
\begin{equation}
\label{aI-bI} a_I=(\de\!x_I^*)(\vol) \quad \textrm{and} \quad
b_I=\vol^*(\de\!x_I).
\end{equation}
Then, by the dual basis lemma, in $\Omega^n(R)$,
\begin{equation}\label{aI-bI formula}
\vol=\sum_{I\in S}a_I\de\!x_I,\; \vol^*=\sum_{I\in
S}b_I(\de\!x_I^*),\; \de\!x_I=b_I\vol, \; \de\!x_I^*=a_I\vol^*.
\end{equation}
It is easy to see that $\sum_{I\in
S}a_Ib_I=1_R.$

In the case that $r > n$, things become more complicated because there is some $s \, (1 \leq s \leq r)$ and $I \in S$ such that
$s\neq I_j$ for all $1\leq j \leq n$. Sometimes we also say $s\notin I$ if $s\neq I_j$ for all $1\leq j \leq n$. Then we have the following lemmas.

\begin{lem}\label{change one}
For any $I=(I_1, I_2, \cdots , I_n)\in S$ and $s\notin I$,
$$\vol (\de\!x_I^*)(\de\! x_s)^*=\sum_{j=1}^n\vol(\de\!x^*_{I_j\rightarrow s})  (\de\!x_{I_j})^*,$$
where $\de\!x^*_{I_j\rightarrow s}$ means
$(\de\!x_{I_j})^*$ is changed to $ (\de\!x_s)^* $ in $\de\!x_I^* = (\de\!x_{I_1})^*\wedge
(\de\!x_{I_2})^*\wedge\cdots \wedge (\de\!x_{I_n})^*$.
\end{lem}
\begin{proof}
Consider the contraction map $\iota_{\vol}: \mathfrak{X}^{n+1}(R)
\to \mathfrak{X}^{1}(R)$ given by $\vol \in \Omega^n(R)$. Then, by
Definition \ref{ctr4},
\begin{align*}
&\iota_{\vol}\big((\de\!x_s)^*\wedge (\de\!x_{I_1})^*\wedge \cdots \wedge (\de\!x_{I_n})^*\big)\\
=&\vol \big((\de\!x_{I_1})^*\wedge \cdots \wedge (\de\!x_{I_n})^*\big)(\de\!x_s)^*\\
&+\sum_{j=1}^n(-1)^j\vol\big((\de\!x_s)^* \wedge (\de\!x_{I_1})^* \wedge \cdots \widehat{(\de\!x_{I_j})^*}
\cdots \wedge (\de\!x_{I_n})^*\big)(\de\!x_{I_j})^*.
\end{align*}
Since $\mathfrak{X}^{n+1}(R)=0$,
$$\vol (\de\!x_I^*)(\de\!x_s)^*=\sum_{j=1}^n(-1)^{j-1}\vol\big((\de\!x_s)^* \wedge (\de\!x_{I_1})^* \wedge \cdots \widehat{(\de\!x_{I_j})^*}
\cdots \wedge (\de\!x_{I_n})^*\big)(\de\!x_{I_j})^*.$$
That is $\vol
(\de\!x_I^*)\de\!x_s^*=\sum_{j=1}^n\vol(\de\!x^*_{I_j\rightarrow s})(\de\!x_{I_j})^*$ by the simplified notation.
\end{proof}
\begin{rmk} The equation in Lemma \ref{change one} still holds
when $\vol$ is changed to any $n$-form in $\Omega^n(R)$. And the condition $s\notin I$ can be removed.
\end{rmk}

\begin{lem}\label{change one duality}
For any $I=(I_1, I_2, \cdots , I_n)\in S$, $a\in R$ and $F\in
\mathfrak{X}^1(R)$,
$$F(a)\de\!x_I=\sum_{j=1}^n(-1)^{j-1}F(x_{I_j})
\de\!a\wedge \de\!x_{I_1}\wedge \de\!x_{I_2}\wedge \cdots
\widehat{\de\!x_{I_j}} \cdots \wedge \de\!x_{I_n}.$$
\end{lem}
\begin{proof}
By Lemma \ref{contraction-formula}(1), $F(a)\de\!x_I=\de\!a\wedge \iota_F(\de\!x_I)$.
Note that $$\iota_F(\de\!x_I)=\sum_{j=1}^n(-1)^{j-1}F(x_{I_j})\de\!x_{I_1}\wedge \de\!x_{I_2}\wedge \cdots \widehat{\de\!x_{I_j}} \cdots \wedge \de\!x_{I_n}.$$
Thus the equation holds.
\end{proof}

Now we give a description of the modular derivation of  $R$ with respect to
the volume form $\vol$.

\begin{thm}\label{modular der of smooth}
Let $R$ be a smooth algebra of dimension $n$ with
trivial canonical bundle $\Omega^n(R)=R\vol$ and a Poisson structure $\{-, -\}$.  Then, with the notations as above, the
modular derivation $\phi_{\vol}$ is given by $$\phi_{\vol}(a)=\sum_{1\leq s\leq
r}(\de\!x_{s})^*(\{a, x_{s}\})+\sum_{I\in S}\{a, a_I\}b_I, \quad \text{for any } a \in R, $$ where
$a_I$ and $b_I$ are defined in \eqref{aI-bI}.
\end{thm}
\begin{proof} By \eqref{aI-bI formula} and Definition \ref{ctr0},
\begin{align*}
\iota_{H_{a}}(\vol)&=\sum_{I\in S}a_I\iota_{H_{a}}(\de\!x_I)\\
&=\sum_{I\in S}a_I\sum_{1\leq j \leq n}(-1)^{j-1}\{a, x_{I_j}\}
\de\!x_{I_1}\wedge \de\!x_{I_2}\wedge \cdots \widehat{\de\!x_{I_j}}
\cdots \wedge \de\!x_{I_n}.
\end{align*}

\begin{align*}
\de\!\iota_{H_{a}}(\vol) =&\sum_{I\in S}\sum_{1\leq j \leq
n}(-1)^{j-1}\{a, x_{I_j}\}\de\!a_I
\wedge \de\!x_{I_1}\wedge \de\!x_{I_2}\wedge \cdots \widehat{\de\!x_{I_j}} \cdots \wedge \de\!x_{I_n}\\
&+\sum_{I\in S}\sum_{1\leq j \leq n}(-1)^{j-1}a_I\de\{a, x_{I_j}\}\wedge \de\!x_{I_1}\wedge \de\!x_{I_2}
\wedge \cdots \widehat{\de\!x_{I_j}} \cdots \wedge \de\!x_{I_n}\\
\stackrel{(a)}{=}&\sum_{I\in S}\{a, a_I\}\de\!x_{I_1}\wedge \de\!x_{I_2}\wedge \dots \wedge\de\!x_{I_n}\\
&+\sum_{I\in S}\sum_{1\leq j \leq n}a_I(\de\!x_{I_j})^*(\{a, x_{I_j}\})\de\!x_{I_1}\wedge \de\!x_{I_2}\wedge \cdots
\wedge\de\!x_{I_j}\wedge\cdots \wedge \de\!x_{I_n}\\
+\sum_{I\in S}\sum_{1\leq j \leq n}& \sum_{s\notin I}
(-1)^{j-1}a_I(\de\!x_{s})^*(\{a, x_{I_j}\})\de\!x_s \wedge
\de\!x_{I_1}\wedge \de\!x_{I_2}\wedge \cdots \widehat{\de\!x_{I_j}}
\cdots \wedge \de\!x_{I_n}.
\end{align*}
where (a) holds by Lemma \ref{change one duality} and $\de\{a, x_{I_j}\}=\sum_{1\leq s \leq r}(\de\!x_{s})^*(\{a, x_{I_j}\})\de\!x_s$.

In order to compute the last term, consider the one-to-one
correspondence on the set of triples $\{(I,j,s)\mid I\in S, 1\leq j
\leq n, s\notin I \}$,
$$ (I\in S, 1\leq j \leq n, s\notin I) \mapsto (I'\in S, 1\leq j' \leq n, s'\notin I')$$
where $I'=(I \backslash \{I_j\})\cup\{s\}= \{I_1, \cdots, \widehat{I_j}, \cdots, I_n, s\}$,
$j'$ is the unique number satisfying $I_{j'-1} < s <I_{j'}$ (i.e. $I'_{j'} = s$) and $s'=I_j$.
Then
\begin{align*}
&\sum_{I\in S}\sum_{1\leq j \leq n}\sum_{s\notin I}(-1)^{j-1}a_I(\de\!x_{s})^*(\{a, x_{I_j}\})
\de\!x_s\wedge \de\!x_{I_1}\wedge \de\!x_{I_2}\wedge \cdots \widehat{\de\!x_{I_j}} \cdots \wedge \de\!x_{I_n}\\
=&\sum_{I'\in S}\sum_{1\leq j' \leq n}\sum_{s'\notin I'}\vol(\de\!x^*_{I'_{j'} \to
{s'}}) (\de\!x_{I'_{j'}})^*(\{a, x_{s'}\})\de\!x_{I'}\\
\stackrel{(b)}{=}&\sum_{I'\in S}\sum_{s'\notin I'}\vol(\de\!x_{I'}^*)(\de\!x_{s'})^*(\{a, x_{s'}\})\de\!x_{I'}\\
=&\sum_{I\in S}\sum_{s\notin I}a_I(\de\!x_{s})^*(\{a,
x_{s}\})\de\!x_{I},
\end{align*}
where (b) holds by Lemma \ref{change one}. So
\begin{align*}
&\de\!\iota_{H_{a}}(\vol)\\
=&\sum_{I\in S}\{a, a_I\}\de\!x_I+\sum_{I\in S}\sum_{1\leq j \leq
n}a_I(\de\!x_{I_j})^*(\{a, x_{I_j}\})\de\!x_{I}
+\sum_{I\in S}\sum_{s\notin I}a_I(\de\!x_{s})^*(\{a, x_{s}\})\de\!x_{I}\\
=&\sum_{I\in S}\{a, a_I\}\de\!x_{I}+\sum_{I\in S}\sum_{1\leq s\leq r}a_I(\de\!x_{s})^*(\{a, x_{s}\})\de\!x_{I}\\
=&\sum_{I\in S}\{a, a_I\}b_I\vol+\sum_{1\leq s\leq r}(\de\!x_{s})^*(\{a, x_{s}\})\vol.
\end{align*}
It follows from the definition of modular derivation that
\begin{align*}
\phi_{\vol}(a)=\sum_{1\leq s\leq r}(\de\!x_{s})^*(\{a, x_{s}\})+\sum_{I\in S}\{a, a_I\}b_I=\phi_1(a)+\phi_2(a),
\end{align*}
where $\phi_1(a)=\sum_{1\leq s\leq r}(\de\!x_{s})^*(\{a, x_{s}\})$, $\phi_2(a)=\sum_{I\in S}\{a, a_I\}b_I$.
\end{proof}
\begin{rmk}
If $r=n$, the set $S$ has only one element $I=(1, 2, \cdots , n)$, then
$\vol=a_I\de\!x_I$ and  $a_Ib_I=1_R$ by \eqref{aI-bI formula}. So
$\phi_2 = \{-, a_I\}b_I=a_I^{-1}\{-, a_I\}$
is a log-Hamiltonian derivation.
In this case, $b_I\vol$ is also a volume form of $R$ and the modular derivation with respect to $b_I\vol$ is $$\phi_{b_I\vol}=\phi_{\vol}-b_I^{-1}H_{b_I}=\phi_{\vol}-\phi_2=\phi_1.$$
\end{rmk}

\begin{exam}
Consider the smooth algebra  $R$ in Example \ref{ep-of -smoth}. Suppose
 $\{-, -\}$ is a Poisson structure over $R$. Then the modular derivation with respect to $\eta=x\de\!y \wedge \de\!z + y\de\!z \wedge \de\!x + z\de\!x \wedge \de\!y$ is given by
 $$\phi_{\eta}(a)= (\de\!x )^*(\{a, x \})+(\de\!y )^*(\{a, y \})+(\de\!z )^*(\{a, z \}),$$
 as $\sum_{I\in S}\{a, a_I\}b_I=\{a,x\}x+\{a,y\}y+\{a,z\}z=0$ for any $a\in R.$
\end{exam}

\section{Poincar\'{e} duality for smooth Poisson algebras}

Any Poisson module can be twisted with a Poisson derivation as given in \cite[Proposition 2.7]{LWW}. In this section,  for any smooth Poisson algebra with trivial canonical bundle, we prove that  the Poisson cochain complex with values in any Poisson module is isomorphic
to the Poisson chain complex with values in the corresponding twisted Poisson module.
Then a version of twisted Poincar\'{e} duality is deduced between Poisson
homologies and Poisson cohomologies, which generalize  \cite[Theorem 3.5]{LWW} for polynomial Poisson algebras.

\subsection{Duality between  derivations and K\"{a}hler differentials}
In this section,  let $R$ be a smooth  affine algebra of dimension $n$ and $\{(\de\!x_i), (\de\!x_i)^{*}\}_{i=1}^r$ be  a dual basis for the projective module $\Omega^1(R)$.

Recall that for any $f_i \in \Hom_R(\Omega^1(R), R)\cong \mathfrak{X}^1(R)$ for $1 \leq i \leq p$, and
$\omega=a_0\de\!a_1 \wedge \de\!a_2 \wedge \cdots \wedge \de\!a_p \in \Omega^p(R)$,
by  Lemma \ref{Xp=Hom-p} and Definition \ref{F-wedge-G},
$$
(f_1 \wedge f_2 \wedge \dots \wedge f_p)(\omega)=
 {a_0 \left|\begin{matrix}
 f_1(\de\!a_1) & f_1(\de\!a_2)& \cdots & f_1(\de\!a_p)\\
 f_2(\de\!a_1) & f_2(\de\!a_2)& \cdots & f_2(\de\!a_p)\\
 \vdots & \vdots & \cdots & \vdots\\
 f_p(\de\!a_1) & f_p(\de\!a_2) & \cdots & f_p(\de\!a_p)\\
\end{matrix}\right|}.
$$
Then $f_1 \wedge f_2 \wedge \dots \wedge f_p \in \mathfrak{X}^p(R)\cong \Hom_R(\Omega^p(R), R).$

Since $\{(\de\!x_i), (\de\!x_i)^{*}\}_{i=1}^r$ is a dual basis for $\Omega^1(R)$,
$\{(\de\!x_i)^*, (\de\!x_i)^{**}\}_{i=1}^r$ is a dual basis
for $ \Hom_R(\Omega^1(R), R) \cong \der(R)$ which is projective.
Then, for any $a \in R$, 
\begin{equation}\label{formula-for-differentials} \de\!a = \sum_{i=1}^r (\de\!x_i)^* (\de\!a)
\de\!x_i = \sum_{i=1}^r (\de\!x_i)^* (a) \de\!x_i,
\end{equation}
and for any $\xi \in \der(R)$,
\begin{equation}\label{formula-for-derivations}
\xi = \sum_{i=1}^r (\de\!x_i)^{**}(\xi)(\de\!x_i)^* =\sum_{i=1}^r
\xi(\de\!x_i)(\de\!x_i)^*=\sum_{i=1}^r \xi(x_i)(\de\!x_i)^*.
\end{equation}
For any $F \in \mathfrak{X}^p(R)$ and fixed  $a_2, \cdots, a_p \in R$,
define $\xi_{a_2, \cdots, a_p}(F) \in \der(R)$ by
$$ \xi_{a_2, \cdots, a_p}(F) (a) \triangleq F(a\wedge a_2\wedge \cdots\wedge a_p).$$
Then, by \eqref{formula-for-derivations}, for any $a \in R$,
\begin{equation}\label{formula-for-derivations-2}
F(a\wedge a_2\wedge \cdots\wedge a_p)= \sum_{i=1}^r \xi_{a_2, \cdots, a_p}(F) (x_i)(\de\!x_i)^*(a)
= \sum_{i=1}^r (\de\!x_i)^*(a) F(x_i\wedge a_2\wedge \cdots\wedge a_p).\end{equation}

The following lemma tells us for any dual basis  $\{\de\!x_i;
(\de\!x_i)^*\}_{i=1}^r$ of $\Omega^1(R)$,  $$\sum_{i=1}^r
(\de\!x_i)^*(\de\!x_i)= n \cdot 1_R,$$ where $n$ is the smooth dimension of $R$.

\begin{lem}\label{dual-basis=n} Let $R$ be a smooth algebra of dimension $n$ with trivial canonical bundle, and $\{\de\!x_i;
(\de\!x_i)^*\}_{i=1}^r$  be a dual basis of the projective
$R$-module $\Omega^1(R)$. Then $\sum_{i=1}^r (\de\!x_i)^*(\de\!x_i)=n \cdot 1_R.$
\end{lem}
\begin{proof} For any $f_i \in \Hom_R(\Omega^1(R), R) (1 \leq i \leq
p)$ and $a=a_0\de\!a_1 \wedge \de\!a_2 \wedge \cdots \wedge \de\!a_p \in
\Omega^p(R)$,
\begin{align*}
&\sum_{i=1}^r(f_1 \wedge f_2\wedge  \cdots\wedge f_p \wedge
(\de\!x_i)^*)
(a_0\de\!a_1 \wedge \de\!a_2 \wedge \cdots \wedge \de\!a_p\wedge \de\!x_i)\\
 =& \sum_{i=1}^r {a_0 \left|\begin{matrix}
 f_1(\de\!a_1) & f_1(\de\!a_2)& \cdots & f_1(\de\!a_p)& f_1(\de\!x_i)\\
 f_2(\de\!a_1) & f_2(\de\!a_2)& \cdots & f_2(\de\!a_p)& f_2(\de\!x_i) \\
 \vdots & \vdots & \cdots & \vdots\\
 f_p(\de\!a_1) & f_p(\de\!a_2) & \cdots & f_p(\de\!a_p) & f_p(\de\!x_i)\\
(\de\!x_i)^*(\de\!a_1) & (\de\!x_i)^*(\de\!a_2)& \cdots & (\de\!x_i)^*(\de\!a_p)& (\de\!x_i)^*(\de\!x_i) \\
\end{matrix}\right|}\\
=&\sum_{i=1}^r (-1)^pa_0 (\de\!x_i)^*(\de\!a_1)(f_1 \wedge f_2 \wedge \dots
\wedge f_p)( \de\!a_2 \wedge \cdots \wedge \de\!a_p  \wedge \de\!x_i)\\
& +\sum_{i=1}^r (-1)^{p+1}a_0 (\de\!x_i)^*(\de\!a_2)(f_1 \wedge f_2 \wedge
\dots\wedge f_p)(\de\!a_1 \wedge \de\!a_3 \wedge \cdots \wedge \de\!a_p  \wedge \de\!x_i)\\
 & + \cdots \\
 &+\sum_{i=1}^r (-1)^{2p-1}a_0 (\de\!x_i)^*(\de\!a_p)(f_1 \wedge f_2 \wedge
\dots\wedge f_p)(\de\!a_1 \wedge \de\!a_2 \wedge \cdots \wedge \de\!a_{p-1}  \wedge \de\!x_i)\\
 & +\sum_{i=1}^r a_0 (\de\!x_i)^*(\de\!x_i)(f_1 \wedge f_2 \wedge
\dots\wedge f_p)(\de\!a_1 \wedge \de\!a_2 \wedge \cdots \wedge \de\!a_p)\\
\stackrel{(a)}{=}&-pa_0(f_1 \wedge f_2 \wedge
\dots\wedge f_p)(\de\!a_1 \wedge \de\!a_2 \wedge \cdots \wedge \de\!a_p)\\
& + \sum_{i=1}^r a_0 (\de\!x_i)^*(\de\!x_i)(f_1 \wedge f_2 \wedge
\dots\wedge f_p)(\de\!a_1 \wedge \de\!a_2 \wedge \cdots \wedge \de\!a_p)  \\
=& a_0 (\sum_{i=1}^r (\de\!x_i)^*(\de\!x_i)-p\cdot 1_R) (f_1 \wedge f_2 \wedge
\dots\wedge f_p)(\de\!a_1 \wedge \de\!a_2 \wedge \cdots \wedge \de\!a_p),
\end{align*}
where $(a)$ holds by \eqref{formula-for-derivations-2}.
It follows that
\begin{align*}
&\sum_{i=1}^r(f_1 \wedge f_2 \wedge \dots \wedge f_p \wedge
(\de\!x_i)^*)
(a_0\de\!a_1 \wedge \de\!a_2 \wedge \cdots \wedge \de\!a_p\wedge \de\!x_i)\\
=&a_0 (\sum_{i=1}^r (\de\!x_i)^*(\de\!x_i)-p\cdot 1_R) (f_1 \wedge f_2 \wedge
\dots\wedge f_p)(\de\!a_1 \wedge \de\!a_2 \wedge \cdots \wedge \de\!a_p).
\end{align*}

Take $p=n$. Note that $\Omega^{n+1}(R)=0$, then it follows that
$$(\sum_{i=1}^r (\de\!x_i)^*(\de\!x_i)-n\cdot 1_R) (f_1 \wedge f_2 \wedge
\dots\wedge f_n)=0.$$
Since $\Omega^n(R)\cong R$,  $\sum_{i=1}^r
(\de\!x_i)^*(\de\!x_i)=n\cdot 1_R$.
\end{proof}

\begin{cor}\label{f-omega}
For any $f \in \Hom_R(\Omega^p(R), R)$ and $\omega \in \Omega^p(R)$,
\begin{align*}
&f(\omega) =\sum_{1\leq i_1<i_2<\cdots <i_{n-p}\leq r}
(f\wedge(\de\!x_{i_1})^*\wedge
 \cdots \wedge(\de\!x_{i_{n-p}})^*)
(\omega\wedge \de\!x_{i_1}\wedge  \cdots \wedge \de\!x_{i_{n-p}})
\end{align*}
\end{cor}
\begin{proof} It follows from the proof of Lemma  \ref{dual-basis=n}
that
\begin{align*}
&\sum_{1\leq i_1,i_2,\cdots ,i_{n-p}\leq
r}(f\wedge(\de\!x_{i_1})^*\wedge \cdots \wedge(\de\!x_{i_{n-p}})^*)
(\omega\wedge \de\!x_{i_1}\wedge  \cdots \wedge \de\!x_{i_{n-p}})\\
=&\sum_{1\leq i_1,i_2,\cdots ,i_{n-p-1}\leq r}(f\wedge
(\de\!x_{i_1})^*\wedge
 \cdots \wedge(\de\!x_{i_{n-p-1}})^*)
(\omega\wedge \de\!x_{i_1}\wedge \cdots \wedge \de\!x_{i_{n-p-1}})\\
=&2!\sum_{1\leq i_1,i_2,\cdots ,i_{n-p-2}\leq r}(f\wedge
(\de\!x_{i_1})^*\wedge
 \cdots \wedge(\de\!x_{i_{n-p-2}})^*)
(\omega\wedge \de\!x_{i_1}\wedge  \cdots \wedge \de\!x_{i_{n-p-2}})\\
=&\cdots\\
=&(n-p)!f(\omega).
\end{align*}
On the other hand,
\begin{align*}
&\sum_{1\leq i_1,i_2,\cdots ,i_{n-p}\leq r}(f\wedge(\de\!x_{i_1})^*\wedge
 \cdots \wedge(\de\!x_{i_{n-p}})^*)
(\omega\wedge \de\!x_{i_1}\wedge  \cdots \wedge \de\!x_{i_{n-p}})\\
=&(n-p)!\sum_{1\leq i_1<i_2<\cdots <i_{n-p}\leq
r}(f\wedge(\de\!x_{i_1})^*\wedge \cdots \wedge(\de\!x_{i_{n-p}})^*)
(\omega\wedge \de\!x_{i_1}\wedge  \cdots \wedge \de\!x_{i_{n-p}}).
\end{align*}
The proof is finished.
\end{proof}

Next, we establish the duality between the derivations and K\"{a}hler differentials for any smooth algebra with trivial canonical bundle  by using the contraction maps.

\begin{thm}\label{dual iso}
Let $R$ be a  smooth algebra of dimension $n$ with trivial canonical bundle. Let $\vol \in \Omega^n(R)$  be a
volume form for $R$. Then the following canonical map is an isomorphism of $R$-modules:
$$\ddag^p_R:\mathfrak{X}^p(R) \to \Omega^{n-p}(R), \, F \mapsto
\iota_F(\vol).$$
\end{thm}
\begin{proof}  
In terms
of the dual basis,
\begin{align*}
&\iota_F(\vol)\\
=&\sum_{1\leq j_{p+1} < \cdots < j_n \leq r} ((\de\!x_{j_{p+1}})^* \wedge \cdots \wedge (\de\!x_{j_n})^*) [\iota_F (\vol)]
\de\!x_{j_{p+1}} \wedge \cdots \wedge \de\!x_{j_n}\\
=&\sum_{1\leq j_{p+1} < \cdots < j_n \leq r}  [\iota_F (\vol)]((\de\!x_{j_{p+1}})^* \wedge \cdots \wedge (\de\!x_{j_n})^*) \,
\de\!x_{j_{p+1}} \wedge \cdots \wedge \de\!x_{j_n}\\
=&\sum_{1\leq j_{p+1} < \cdots < j_n \leq r} (F \wedge (\de\!x_{j_{p+1}})^* \wedge \cdots \wedge (\de\!x_{j_n})^* )(\vol ) \,
\de\!x_{j_{p+1}} \wedge \cdots \wedge \de\!x_{j_n}.
\end{align*}

Consider the map $\flat^{n-p}_R: \Omega^{n-p}(R)  \to \mathfrak{X}^p(R), \, \omega \mapsto \iota_{\omega} (\vol^*).$
In terms of the dual basis,
\begin{align*}
&\iota_{\omega} (\vol^*)\\
=&\sum_{1 \leq j_1 < \cdots < j_p \leq r} \iota_{\omega}(\vol^*) (\de\!x_{j_1} \wedge \cdots \wedge \de\!x_{j_p})\,
 (\de\!x_{j_1})^*
\wedge (\de\!x_{j_2})^* \wedge \cdots \wedge (\de\!x_{j_p})^*\\
=&\sum_{1 \leq j_1 < \cdots < j_p \leq r} \vol^* (\de\!x_{j_1} \wedge \cdots \wedge \de\!x_{j_p}
\wedge \omega) \, (\de\!x_{j_1})^*
\wedge (\de\!x_{j_2})^* \wedge \cdots \wedge (\de\!x_{j_p})^*.
\end{align*}
Then
\begin{align*}
& (\ddag^p_R \flat^{n-p}_R)(\omega)\\
=&\sum_{1 \leq j_1 < \cdots < j_p \leq r} \sum_{1 \leq l_{p+1} < \cdots < l_n \leq r}  \vol^* (\de\!x_{j_1} \wedge \cdots \wedge \de\!x_{j_p} \wedge \omega)\\
& \quad ((\de\!x_{j_1})^*\wedge \cdots \wedge (\de\!x_{j_p})^* \wedge (\de\!x_{l_{p+1}})^* \wedge \cdots \wedge (\de\!x_{l_n})^*)(\vol)\,\de\!x_{l_{p+1}} \wedge \cdots \wedge \de\!x_{l_n}\\
=&\sum_{1 \leq l_{p+1} < \cdots < l_n \leq r} \sum_{1\leq j_1 < \cdots < j_p \leq r} ((\de\!x_{j_1})^* \wedge
\cdots \wedge (\de\!x_{j_p})^* \wedge (\de\!x_{l_{p+1}})^* \wedge \cdots
\wedge (\de\!x_{l_n})^*) \\
& \quad (\de\!x_{j_1} \wedge \cdots \wedge \de\!x_{j_p} \wedge \omega) \,\de\!x_{l_{p+1}} \wedge \cdots \wedge \de\!x_{l_n}  \\
=& \sum_{1 \leq l_{p+1} < \cdots < l_n \leq r}((\de\!x_{l_{p+1}})^* \wedge \cdots
\wedge (\de\!x_{l_n})^*)(\omega) \, \de\!x_{l_{p+1}} \wedge \cdots \wedge \de\!x_{l_n} \\
=& \omega.
\end{align*}
The second last equality holds by Corollary  \ref{f-omega}. On the other hand,
\begin{align*}
&(\flat^{n-p}_R  \ddag^p_R)(F)\\
  =&\sum_{1 \leq j_{p+1} < \cdots < j_n \leq r}\sum_{1 \leq l_1 < \cdots < l_p \leq r} (F
\wedge (\de\!x_{j_{p+1}})^* \wedge \cdots \wedge (\de\!x_{j_n})^*) (\vol)  \\
& \quad  \vol^* (\de\!x_{l_1} \wedge \cdots \wedge \de\!x_{l_p} \wedge
\de\!x_{j_{p+1}} \wedge \cdots \wedge \de\!x_{j_n}) \, (\de\!x_{l_1})^* \wedge
\cdots \wedge (\de\!x_{l_p})^*\\
=& \sum_{1 \leq l_1 < \cdots < l_p \leq r} F(\de\!x_{l_1} \wedge \cdots \wedge \de\!x_{l_p})
  \, (\de\!x_{l_1})^* \wedge
\cdots \wedge (\de\!x_{l_p})^*\\
=& F.
\end{align*}
So, $\ddag^p_R$ is an isomorphism with the inverse $\flat^{n-p}_R$.
\end{proof}

\begin{thm}\label{dagM and dagR}
Let $R$ be a smooth algebra of dimension $n$ with
trivial canonical bundle, and $\vol \in \Omega^n(R)$  be a
volume form for $R$. Then for any $R$-module $M$ and $p\in
\mathbb{N}$, the canonical map
$$\ddag^p_M: \mathfrak{X}^p(M) \to M
\otimes_R \Omega^{n-p}(R), \, F \mapsto \iota_F(\vol)$$ is the
composition of the following canonical isomorphisms
\begin{align*}
 \mathfrak{X}^p(M) & \cong \Hom_R(\Omega^p(R), M) \cong    M \otimes_R \Hom_R(\Omega^p(R), R)\\
 & \cong M \otimes_R  \mathfrak{X}^p(R)\stackrel{\id_M \otimes \ddag^p_R} {\cong}  M \otimes_R \Omega^{n-p}(R).
 \end{align*}
\end{thm}

\begin{proof}
Since  $\Omega^{p}(R)$ is a finitely generated projective
$R$-module, the canonical map $\Hom_R(\Omega^p(R), M)
 \to   M \otimes_R \Hom_R(\Omega^p(R), R)$,
 $$ g \mapsto \sum_{1 \leq i_1 < \cdots < i_p \leq r}
 g(\de\!x_{i_1} \wedge \cdots \wedge \de\!x_{i_p}) \otimes (\de\!x_{i_1})^* \wedge \cdots \wedge (\de\!x_{i_p})^* $$
is an isomorphism by the dual basis lemma. Then the composition map is
$$F \mapsto  \sum_{1 \leq i_1 < \cdots < i_p \leq r} F(x_{i_1}\wedge \cdots \wedge x_{i_p}) \otimes
\iota_{(\de\!x_{i_1})^* \wedge \cdots \wedge (\de\!x_{i_p})^*
}(\vol).$$
For any $n$-form $\de\!a_1\wedge  \de\!a_2 \wedge \cdots
\wedge \de\!a_n \in \Omega^n(R)$,
\begin{align*}
&\sum_{1 \leq i_1 < \cdots < i_p \leq r} F(x_{i_1} \wedge \cdots \wedge x_{i_p})
\otimes \iota_{(\de\!x_{i_1})^* \wedge \cdots \wedge (\de\!x_{i_p})^* }(\de\!a_1\wedge \de\!a_2 \wedge \cdots \wedge \de\!a_n)\\
=&\sum_{1 \leq i_1 < \cdots < i_p \leq r} F(x_{i_1} \wedge \cdots \wedge x_{i_p}) \otimes  \sum_{\sigma \in S_{p, n-p}} \sgn(\sigma)\\
& \quad \big((\de\!x_{i_1})^* \wedge \cdots \wedge (\de\!x_{i_p})^*\big) (\de\!a_{\sigma(1)}\wedge  \cdots \wedge \de\!a_{\sigma(p)}) \,
\de\!a_{\sigma(p+1)} \wedge \cdots \wedge \de\!a_{\sigma(n)}\\
=&\sum_{1 \leq i_1 < \cdots < i_p \leq r} \sum_{\sigma \in S_{p,
n-p}} \sgn(\sigma) F(\de\!x_{i_1}\wedge \cdots \wedge
\de\!x_{i_p})\\
&\quad \big((\de\!x_{i_1})^* \wedge \cdots \wedge (\de\!x_{i_p})^*\big)
(\de\!a_{\sigma(1)}\wedge \cdots \wedge \de\!a_{\sigma(p)}) \otimes
 \de\!a_{\sigma(p+1)}\wedge \cdots \wedge \de\!a_{\sigma(n)}\\
=&\sum_{\sigma \in S_{p, n-p}} \sgn(\sigma) \sum_{1 \leq i_1 <
\cdots < i_p \leq r}\\
&\quad F [\big((\de\!x_{i_1})^* \wedge \cdots \wedge (\de\!x_{i_p})^*\big)
(\de\!a_{\sigma(1)}\wedge \cdots \wedge \de\!a_{\sigma(p)})
\de\!x_{i_1}\wedge \cdots \wedge
\de\!x_{i_p}]\\
& \quad \quad \otimes
 \de\!a_{\sigma(p+1)}\wedge \cdots \wedge \de\!a_{\sigma(n)}\\
=&\sum_{\sigma \in S_{p, n-p}} \sgn(\sigma)
F(\de\!a_{\sigma(1)}\wedge \cdots \wedge \de\!a_{\sigma(p)}) \otimes
\de\!a_{\sigma(p+1)}\wedge \cdots \wedge \de\!a_{\sigma(n)}\\
=&\iota_F(\de\!a_1\wedge \de\!a_2 \wedge \cdots \wedge \de\!a_n).
\end{align*}
It follows that $$\sum_{1 \leq i_1 < \cdots < i_p \leq r} F(x_{i_1} \wedge
\cdots \wedge x_{i_p}) \otimes \iota_{(\de\!x_{i_1})^* \wedge \cdots
\wedge (\de\!x_{i_p})^* }(\vol) =\iota_F(\vol).$$
So, $\ddag^p_M$ is the composition of the isomorphisms as stated.

The inverse map of  $\ddag^p_M$ is given by the composition
$$M \otimes_R \Omega^{n-p}(R) \to M \otimes_R \mathfrak{X}^p(R)
\to \mathfrak{X}^p(M),$$
\begin{align*}
 m \otimes \omega \mapsto  m \otimes \iota_{\omega}(\vol^*)
 \mapsto [a_1 \wedge \cdots \wedge a_p \mapsto  m \cdot \vol^* (\de\!a_1 \wedge \cdots \wedge \de\!a_p \wedge \omega)].
\end{align*}
\end{proof}

\subsection{Poincar\'{e} duality between Poisson homology and cohomology}
Assume that $R$ is a smooth Poisson algebra of dimension $n$ with trivial canonical bundle. By  \cite[Proposition 2.7]{LWW}, any Poisson module can be twisted by a Poisson derivation.
The main result in this section
is that there is a twisted Poincar\'{e} duality $\PH^*(R, M) \cong
\PH_{n-*}(R, M_t)$ for any Poisson $R$-module $M$, where $M_t$ is the
twisted Poisson module of $M$ twisted by the modular derivation of
$R$. This generalizes the main results in \cite{LR,Zhu,LWW,LWZ20}.
If the Poisson structure of $R$ is unimodular, then it reduces to
the classical Poincar\'{e} duality $\PH^*(R, M) \cong \PH_{n-*}(R, M)$ for any
Poisson $R$-module $M$.

To avoid confusion, let $\partial^M$ and $\delta_M$ denote the differentials of the Poisson cochain complex and Poisson chain complex of $R$ with values in $M$ respectively.
To simplify, let $\partial=\partial^R$ and $\delta=\delta_R$.


\begin{lem} \label{diff-codiff1}
Let $M$ be a  right Poisson $R$-module. For any $ F\in \mathfrak{X}^p(R)$,
$$\iota_F \partial^M  - (-1)^p \partial^M  \iota_F =\iota_{\delta F}: M \otimes_R \Omega^{*}(R) \to M \otimes_R  \Omega^{*-p-1}(R).$$
\end{lem}
\begin{proof}
If $q\leq p$, then, obviously, for any $\omega \in M \otimes_R \Omega^{q}(R)$, $\iota_F(\partial^M \omega), \partial^M (\iota_F \omega)$ and $\iota_{\delta F}(\omega)$ are all $0$.
Now suppose $q>p$ and $\omega=m\otimes \de\!a_1\wedge \cdots \wedge \de\!a_q$. Then,
\begin{align*}
&\iota_F(\partial^M \omega)\\
=&\iota_F(\sum_{i=1}^q(-1)^{i-1}\{m, a_i\}_M\otimes \de\!a_1\wedge \cdots \widehat{\de\!a_i} \cdots \wedge \de\!a_q) \\
& {}+\iota_F(\sum_{1\le i<j\le q}(-1)^{i+j}m \otimes \de \{a_i, a_j\}\wedge \de\!a_1\wedge \cdots  \widehat{\de\!a_i} \cdots  \widehat{\de\!a_j} \cdots \wedge \de\!a_q)\\
=&\sum\limits_{\sigma\in S_{1,p,q-p-1}} \sgn(\sigma)\{m, a_{\sigma(1)}\}_M F(a_{\sigma(2)}\wedge\cdots\wedge a_{\sigma(p+1)})\otimes \de\! a_{\sigma(p+2)}\wedge \cdots\wedge\de\! a_{\sigma(q)}\\
& -\sum\limits_{\sigma\in S_{2,p-1,q-p-1}}\sgn(\sigma)\, m F(\{a_{\sigma(1)}, a_{\sigma(2)}\}\wedge a_{\sigma(3)}\wedge\cdots\wedge a_{\sigma(p+1)}) \otimes\\
& \quad  \quad  \quad\de\! a_{\sigma(p+2)}\wedge \cdots\wedge\de\! a_{\sigma(q)}\\
& -\sum\limits_{\sigma\in S_{p,2,q-p-2}}(-1)^p\sgn(\sigma)\,m F(a_{\sigma(1)}\wedge \cdots\wedge a_{\sigma(p)})\otimes\\
 &\quad \quad \quad \de \{a_{\sigma(p+1)}, a_{\sigma(p+2)}\}\wedge \de\!a_{\sigma(p+3)}\wedge \cdots\wedge\de\! a_{\sigma(q)},
\end{align*}
\begin{align*}
&\partial^M (\iota_F \omega)\\
=&\partial^M \big(\sum\limits_{\sigma\in S_{p,q-p}}
 \sgn(\sigma)\, m F(a_{\sigma(1)}\wedge\cdots\wedge a_{\sigma(p)})\otimes
 \de\! a_{\sigma(p+1)}\wedge \de\! a_{\sigma(p+2)}\wedge\cdots\wedge\de\! a_{\sigma(q)}\big)\\
 =&\sum\limits_{\sigma\in S_{p,1,q-p-1}}
 \sgn(\sigma)\{m F(a_{\sigma(1)}\wedge\cdots\wedge a_{\sigma(p)}), a_{\sigma(p+1)}\}_M\otimes
 \de\! a_{\sigma(p+2)}\wedge \cdots\wedge\de\! a_{\sigma(q)}\\
&-\sum\limits_{\sigma\in S_{p,2,q-p-2}} \sgn(\sigma)\,m F(a_{\sigma(1)}\wedge\cdots\wedge a_{\sigma(p)})\otimes\\
&\quad \quad \de\!\{ a_{\sigma(p+1)}, a_{\sigma(p+2)}\}\wedge \de\! a_{\sigma(p+3)}\wedge \cdots\wedge\de\! a_{\sigma(q)}\\
 =&\sum\limits_{\sigma\in S_{p,1,q-p-1}}
 \sgn(\sigma)\,m \{F(a_{\sigma(1)}\wedge\cdots\wedge a_{\sigma(p)}), a_{\sigma(p+1)}\}\otimes
 \de\! a_{\sigma(p+2)}\wedge \cdots\wedge\de\! a_{\sigma(q)}\\
& + \sum\limits_{\sigma\in S_{p,1,q-p-1}}
 \sgn(\sigma)\{m, a_{\sigma(p+1)}\}_M F(a_{\sigma(1)}\wedge\cdots\wedge a_{\sigma(p)})\otimes
 \de\! a_{\sigma(p+2)}\wedge \cdots\wedge\de\! a_{\sigma(q)}\\
&-\sum\limits_{\sigma\in S_{p,2,q-p-2}}
 \sgn(\sigma)\,m F(a_{\sigma(1)}\wedge\cdots\wedge a_{\sigma(p)})\otimes\\
 &\quad \quad \de\!\{ a_{\sigma(p+1)}, a_{\sigma(p+2)}\}\wedge
 \de\! a_{\sigma(p+3)}\wedge \cdots\wedge\de\! a_{\sigma(q)}\\
 =&(-1)^p\sum\limits_{\sigma\in S_{1,p,q-p-1}}
 \sgn(\sigma)\,m \{F(a_{\sigma(2)}\wedge\cdots\wedge a_{\sigma(p+1)}), a_{\sigma(1)}\}\otimes
 \de\! a_{\sigma(p+2)}\wedge \cdots\wedge\de\! a_{\sigma(q)}\\
& + (-1)^p\sum\limits_{\sigma\in S_{1,p,q-p-1}}\sgn(\sigma)\{m, a_{\sigma(1)}\}_M F(a_{\sigma(2)}\wedge\cdots\wedge a_{\sigma(p+1)})\otimes \\
&\quad  \quad \de\! a_{\sigma(p+2)}\wedge \cdots\wedge\de\! a_{\sigma(q)}\\
&- \sum\limits_{\sigma\in S_{p,2,q-p-2}}
 \sgn(\sigma)\,m F(a_{\sigma(1)}\wedge\cdots\wedge a_{\sigma(p)})\otimes\\
&\quad \quad \de\!\{ a_{\sigma(p+1)}, a_{\sigma(p+2)}\}\wedge \de\! a_{\sigma(p+3)}\wedge \cdots\wedge\de\! a_{\sigma(q)},
\end{align*}
and
\begin{align*}
&\iota_{\delta F}(\omega)\\
=&\sum\limits_{\sigma\in S_{p+1,q-p-1}}
 \sgn(\sigma)\, m (\delta F)(a_{\sigma(1)}\wedge\cdots\wedge a_{\sigma(p+1)})\otimes
 \de\! a_{\sigma(p+2)}\wedge \cdots\wedge\de\! a_{\sigma(q)}\\
 =&-\sum\limits_{\sigma\in S_{1,p,q-p-1}} \sgn(\sigma)\, m \{F(a_{\sigma(2)}\wedge\cdots\wedge a_{\sigma(p+1)}), a_{\sigma(1)}\}\otimes
 \de\! a_{\sigma(p+2)}\wedge \cdots\wedge\de\! a_{\sigma(q)}\\
 &-\sum\limits_{\sigma\in S_{2,p-1,q-p-1}}\sgn(\sigma)\, m F(\{a_{\sigma(1)}, a_{\sigma(2)}\}\wedge \cdots\wedge a_{\sigma(p+1)}) \otimes \de\! a_{\sigma(p+2)}\wedge \cdots\wedge\de\! a_{\sigma(q)}.
\end{align*}
It follows that  $\iota_F(\partial^M \omega)- (-1)^p \partial^M (\iota_F \omega)=\iota_{(\delta F)}\omega$ for any $\omega \in M \otimes_R \Omega^{q}(R)$.
\end{proof}

Similarly, the following lemma holds.
\begin{lem} \label{diff-codiff2}
Let $M$ be a  right Poisson   $R$-module.
For any $F\in \mathfrak{X}^p(M),$
$$\iota_F \partial -(-1)^p \partial^M \iota_F =\iota_{\delta_M F}:  \Omega^{*}(R) \to M \otimes_R  \Omega^{*-p-1}(R).$$
\end{lem}

\begin{lem}  \label{wedge cocycle}
Let $M$ be a  right Poisson   $R$-module. Then for any $F\in \mathfrak{X}^p(R)$ and  $G\in \mathfrak{X}^q(M)$,
$$\delta_M(F\wedge G)=\delta F\wedge G+(-1)^p F\wedge \delta_M G.$$
\end{lem}
\begin{proof}
It is easy to check from the definitions. The reader may also refer to \cite[Proposition 3.7]{LPV} for more details.
\end{proof}

\begin{rmk}\label{PH-is-Walg}
By Lemma  \ref{wedge cocycle}, $\wedge$ is well-defined on $\PH^*(R)$. Hence, $(\PH^*(R), \wedge)$ is a graded commutative algebra, and
$(\PH^*(R, M), \wedge)$ is a $(\PH^*(R), \wedge)$-module.
\end{rmk}
\begin{lem} \cite[Lemma 3]{LWW1} \label{twist diff}
Let $M$ be a  right Poisson   $R$-module, and $M_\phi$ be the twisted Poisson module of $M$ twisted by a Poisson derivation $\phi$, that is,
$\{m, a\}_{M_\phi} = \{m, a\}_M + m \cdot \phi (a)$ for any $m \in M$ and $a \in R$ (See \cite[Proposition 2.7]{LWW}).
Then $$\delta_{M_{\phi}}=\delta_{M}-(\phi \wedge\text{-}) \quad \textrm{and} \quad \partial^{M_{\phi}}=\partial^{M}+\iota_{\phi}.$$
\end{lem}
The proof follows directly from the definitions of Poisson (co)chain complex and the twisted module structure.

\begin{prop} \label{vol cycle}
Let $(R, \pi)$ be a  smooth Poisson  algebra of dimension $n$ with trivial canonical bundle,
$\vol \in \Omega^n(R) $ be a volume form and $\phi_{\vol}$ be the modular derivation
with respect to $\vol$.
Let $R_t$ be the
twisted Poisson module of $R$ twisted by  $\phi_{\vol}$.
Then $\vol$ is a Poisson cycle in the Poisson  complex of $R$ with values in the module $R_t$, i.e., $\partial \vol = -\iota_{\phi_{\vol}} \vol$.
\end{prop}
\begin{proof}
By Lemma  \ref{twist diff}, $\partial^{R_t}=\partial+\iota_{\phi_{\vol}}=[\iota_{\pi}, \de]+\iota_{\phi_{\vol}}.$
Then $\partial^{R_t}(\vol)=-\de \iota_{\pi}(\vol)+\iota_{\phi_{\vol}}(\vol)\in  \Omega^{n-1}(R)$.
It suffices to prove $\iota_{\de\iota_{\pi}(\vol)} \vol^*=\phi_{\vol}\in \mathfrak{X}^1(R)$ by Theorem  \ref{dual iso}.

For any $a\in R$, by Lemma  \ref{contraction-formula}, $\iota_{H_a}(\vol)=-\de\! a\wedge
\iota_{\pi}(\vol)$.
Then $\de \iota_{H_a}(\vol)=\de\! a\wedge
\de \iota_{\pi}(\vol)$. So
\begin{align*}
(\iota_{\de\iota_{\pi}(\vol)} \vol^* ) (a)\vol=&\vol^*(\de\!a \wedge \de\iota_{\pi}(\vol))\vol &\mbox{(by Definition  \ref{ctr2})}\\
=& \de\!a \wedge \de\iota_{\pi}(\vol)&\\
=&\de \iota_{H_a}(\vol)&\\
=&\phi_{\vol}(a)\vol.&\mbox{(by Definition  \ref{modular4})}
\end{align*}
 Hence  $(\iota_{\de\iota_{\pi}(\vol)} \vol^* ) (a)=\phi_{\vol}(a)$. Since $a$ is arbitrary, $\iota_{\de\iota_{\pi}(\vol)} \vol^*=\phi_{\vol}$, which ends the proof.
\end{proof}

\begin{thm}\label{comm diag1}
Let $R$ be a smooth Poisson  algebra of dimension $n$ with trivial
canonical bundle, $\vol \in \Omega^n(R) $ be a volume form and
$\phi_{\vol}$ be the modular derivation of $R$ with respect to
$\vol$. Let $M$ be a Poisson $R$-module and  $M_t$ be the twisted
Poisson module of $M$ twisted by  $\phi_{\vol}$, i.e.,
\begin{equation}\label{twist-module}
\{m, a\}_{M_t} = \{m, a\}_M + m \cdot \phi_{\vol}(a),  \, \forall
\,m \in M, a \in R.
\end{equation}
Then the following diagram 
$$
\xymatrix{
   \mathfrak{X}^p(M) \ar[d]_{\dag^p_M}^{\cong} \ar[r]^{{\delta}_M}
        &   \mathfrak{X}^{p+1}(M)  \ar[d]^{\dag^{p+1}_M}_{\cong}    \\
  M_t \otimes_R \Omega^{n-p}(R)    \ar[r]^{{\partial}^{M_t}}
        & M_t \otimes_R \Omega^{n-p-1}(R)
     }
$$
is commutative, where $\dag^p_M=(-1)^{\frac{p(p+1)}{2}}\ddag^p_M$, and $\ddag^p_M$ is the isomorphism in Theorem  \ref{dagM and dagR}.
\end{thm}
\begin{proof}
For any $F\in \mathfrak{X}^p(M)$, $\ddag^p_M(F)=\iota_F \vol$,
\begin{align*}
&\partial^{M_t}(\iota_F \vol)\\
=&\partial^{M}(\iota_F \vol)+\iota_{\phi_{\vol}}(\iota_F \vol)&\mbox{(by Lemma  \ref{twist diff})}\\
=&(-1)^p\big(\iota_F(\partial \vol)-\iota_{(\delta_M F)}\vol\big)+\iota_{\phi_{\vol}}(\iota_F \vol)&\mbox{(by Lemma  \ref{diff-codiff2})}\\
=&(-1)^{p-1}\iota_F(\iota_{\phi_{\vol}} \vol)+(-1)^{p-1}\iota_{(\delta_M F)}\vol+\iota_{\phi_{\vol}}(\iota_F \vol)&\mbox{(by Lemma  \ref{vol cycle})}\\
=&(-1)^{p-1}\iota_{(\delta_M F)}\vol &\mbox{(by Proposition  \ref{ctr com})}\\
=&(-1)^{p-1}\ddag^p_M(\delta_M F).
\end{align*}
\end{proof}

The following theorem  follows from Theorems \ref{dagM and dagR} and \ref{comm diag1} directly.

\begin{thm}\label{main-theorem1}
Let $R$ be a smooth Poisson algebra of dimension $n$ with trivial
canonical bundle,  $M$ be a Poisson $R$-module and  $M_t$ be the twisted
Poisson module of $M$ twisted by  the modular derivation.
Then for any $p \in \mathbb{N}$,  $$\PH_p(R, M_t)\cong \PH^{n-p}(R, M).$$
\end{thm}

\begin{rmk}\label{unimodular-com-dia} For a unimodular  Poisson algebra  $R$,
its modular derivation $\phi_{\vol}$ may be not zero but a log-Hamiltonian derivation. Then we can choose another volume form $\vol' \in \Omega^n(R) $ such that the modular derivation of $R$ with respect to $\vol'$ is zero. Hence, the diagram in this case
$$
\xymatrix{
   \mathfrak{X}^p(M) \ar[d]_{\dag'^p_M}^{\cong} \ar[r]^{{\delta}_M}
        &   \mathfrak{X}^{p+1}(M)  \ar[d]^{\dag'^{p+1}_M}_{\cong}    \\
  M \otimes_R \Omega^{n-p}(R)    \ar[r]^{{\partial}^{M}}
        & M \otimes_R \Omega^{n-p-1}(R)
     }
$$
 is commutative, where $\dag'_M$ is induced by $\vol'$. It follows that  $$\PH_p(R, M)\cong \PH^{n-p}(R, M)$$ for any $p \in \mathbb{N}$.
\end{rmk}

\section {Batalin-Vilkovisky structure on Poisson cohomology}
In this section, we study the Batalin-Vilkovisky algebra structure on the Poisson cochain complex and  its cohomology, where the Batalin-Vilkovisky operator is induced from the de Rham differential on the Poisson chain complex via Poincar\'{e} duality. We first list some definitions about Gerstenhaber algebras and  Batalin-Vilkovisky algebras. 
\subsection{Gerstenhaber algebra and  Batalin-Vilkoviksy algebra}
%
\begin{defn} 
A {\it Gerstenhaber algebra} is a  graded-commutative  algebra
$(G=\bigoplus_{i\in \mathbb{Z}} G_i, \cdot)$ together with a bracket 
$$[-, -]: G\times G\rightarrow G,$$ called the Gerstenhaber bracket of $G$, such that 
\begin{enumerate}
\item the shift $G(1)$ is a graded Lie algebra of degree $0$;
\item for any homogeneous elements $a, b, c \in G$,\begin{equation}\label{derivation-identity}
[a, b \cdot c] = [a, b] \cdot c + (-1)^{(|a|-1)|b|}b \cdot [a, c].
\end{equation}
\end{enumerate}


\end{defn}

\begin{exam}\label{general-G-stru}
  For any commutative algebra $R$, $(\mathfrak{X}^*(R), \wedge, [- ,-]_{\bf SN})$ is a Gerstenhaber algebra (see \cite[Proposition 3.7]{LPV}), where $[- ,-]_{\bf SN}: \mathfrak{X}^p(R) \times \mathfrak{X}^q(R)\rightarrow \mathfrak{X}^{p+q-1}(R)$ is the {\it Schouten-Nijenhuis} bracket : for any $P\in \mathfrak{X}^p(R)$ and $Q\in \mathfrak{X}^q(R)$, \begin{align*}
  &[P, Q]_{\bf SN}(a_1\wedge a_2\wedge \cdots\wedge a_{p+q-1})\\
  =&(-1)^{(p-1)(q-1)}\sum_{\sigma\in S_{q,p-1}}\sgn(\sigma)P(Q(a_{\sigma(1)}\wedge  \cdots\wedge  a_{\sigma(q)})\wedge  a_{\sigma(q+1)}\wedge  \cdots\wedge  a_{\sigma(p+q-1)})\\
  &-\sum_{\sigma\in S_{p,q-1}}\sgn(\sigma)Q(P(a_{\sigma(1)}\wedge  \cdots\wedge a_{\sigma(p)})\wedge  a_{\sigma(p+1)}\wedge  \cdots\wedge  a_{\sigma(p+q-1)}).
\end{align*}
\end{exam}
For an algebra $R$, recall that an element $\pi \in \mathfrak{X}^2(R)$ is a Poisson bracket iff $[\pi,\pi]_{\bf SN}=0$.  And the Poisson cohomology differential $\delta$ is given by the
Schouten bracket $[\pi, -]_{\bf SN}$. Hence, one can obtain the following conclusion. 
\begin{exam}(\cite[Proposition 4.9]{LPV})
For any Poisson algebra $R$, its Poisson cohomology $(\PH^*(R), \wedge, [- ,-]_{\bf SN})$ is also a Gerstenhaber algebra.
\end{exam}

\begin{exam}\label{general-G-stru-2}
Let $R$ be a commutative algebra. Then $(\wedge^*_R(\Der(R)), \wedge, [- ,-])$ is a Gerstenhaber algebras (see
\cite[Section 6.1]{Gin}),
where $[- ,-]$   is defined as follows: for any $p, q \in
\mathbb{N}$,
\begin{gather*}
 [-,-]: \wedge^p_R\der(R) \times \wedge^q_R\der(R)\rightarrow \wedge^{p+q-1}_R\Der(R), \\
[\xi_1\wedge\cdots\wedge\xi_p, \eta_1\wedge\cdots\wedge\eta_q]
 \triangleq  \sum_{i=1}^p \sum_{j=1}^q(-1)^{i+j}[\xi_i,
\eta_j]\wedge\xi_1\cdots\widehat{\xi_i}\cdots\wedge\xi_p\wedge\eta_1\cdots\widehat{\eta_j}\cdots\wedge\eta_q.
\end{gather*}
\end{exam}

\begin{rmk}\label{schouten-bracket-under-identification}
Suppose that $R$ is a smooth algebra. Then
$$(\mathfrak{X}^*(R), \wedge, [- ,-]_{\bf SN})\cong (\wedge^*_R(\Der(R)), \wedge, [- ,-])$$
as Gerstenhaber algebras, where the isomorphism  $\alpha$  is defined in Corollary \eqref{smooth duality}. Moreover, if $\pi \in \mathfrak{X}^2(R)$ is a Poisson bracket, then the following diagram 
 \begin{equation*}
 \xymatrix@C=2cm{
  \wedge_R^p\der(R) \ar[d]_{\alpha} \ar[r]^{[\pi,-]}
         & \wedge_R^{p+1}\der(R) \ar[d]^{\alpha} \\
 \mathfrak{X}^p(R)  \ar[r]^{{\delta}^p=[\pi,-]_{\bf SN}} &   \mathfrak{X}^{p+1}(R)
 }
 \end{equation*}
is commutative for any $p \in
\mathbb{N}$.

\end{rmk}


\begin{defn}
Let $(V, \cdot)$ be a graded-commutative graded algebra. A
{\it Batalin-Vilkovisky operator }$\Delta$ on $V$ is an operator $\Delta:
V \rightarrow  V$ of degree $-1$ such that $\Delta^2=0$ and the
obstruction of $\Delta$ from being a graded-derivation
\begin{align}\label{difference}
[a, b] \triangleq (-1)^{|a|}(\Delta(a\cdot b)-\Delta(a)\cdot b- (-1)^{|a|}a \cdot \Delta(b))
\end{align}
is a graded-derivation, that is, \eqref{derivation-identity} holds.
The triple $(V, \cdot, \Delta)$ is called a {\it Batalin-Vilkovisky
algebra} (BV algebra, for short).
\end{defn}

\begin{rmk}
Any BV algebra is a Gerstenhaber algebra if one
defines the Gerstenhaber bracket by \eqref{difference}.
A Gerstenhaber algebra with the bracket $[-,-]$ is a BV algebra (or said to be exact) if it can be
equipped with an operator  $\Delta$ of degree $-1$ such that $\Delta^2= 0$ and $[-,-]$
measures the deviation of $\Delta$ from being a derivation, i.e.,  \eqref{difference} holds.
\end{rmk}

In the following, we will focus
on the BV algebra structure on Poisson cohomology.

\subsection{BV structure on $\mathfrak{X}^*(R)$}
If $R$ is a smooth algebra with trivial canonical bundle, i.e., $R$ satisfies the condition of Theorem  \ref{comm diag1}, then
one can define an operator $\Delta$ on $\mathfrak{X}^*(R)$ induced by the de Rham differential $\de$ as follows:
\begin{equation}\label{def-of-bv-operator}
\xymatrix{
   \mathfrak{X}^{p}(R) \ar[d]_{\dag_R^{p}} \ar@{-->}[r]^{\Delta}
        &   \mathfrak{X}^{p-1}(R)      \\
  \Omega^{n-p}(R)    \ar[r]^{\de}
        &  \Omega^{n-p+1}(R)\ar[u]_{(\dag_R^{p-1})^{-1}}.
     }
\end{equation}
Then $\Delta$ is of degree $-1$ and  $\Delta^2=0$.  In order to show that $\Delta$ is a BV operator on $\mathfrak{X}^*(R)$, it
suffices to check that the identity \eqref{difference} holds when
$[-,-]$ is the Gerstenhaber bracket defined in Example \ref{general-G-stru-2}. This follows from the  the following Lemmas \ref{con-map-wedge} to \ref{key-lemma}.
\begin{lem}\label{con-map-wedge}
If $P\in \mathfrak{X}^p(R)$, $Q\in \mathfrak{X}^q(R)$ and $\omega\in \Omega^{p+q-1}(R)$, then, in $\mathfrak{X}^1(R)$,
$$\iota_\omega(P\wedge Q)=(-1)^{(p-1)q}\iota_\alpha(P)+(-1)^p\iota_\beta(Q),$$
where $\alpha=\iota_Q(\omega)$, $\beta=\iota_P(\omega)$.
\end{lem}

\begin{proof}
Without loss of generality, suppose $\omega=a_0\de\!a_1\wedge\cdots\wedge\de\!a_{p+q-1}\in
\Omega^{p+q-1}(R)$. Then the equation holds via an explicit computation by definitions.
\end{proof}


\begin{lem}\label{a}
If $\omega \in
\Omega^p(R)$ and $\xi_1, \cdots, \xi_{p+1}\in \mathfrak{X}^1(R)$, then
\begin{align*}
(\xi_1\wedge\cdots\wedge\xi_{p+1})(\de\omega)
=&\sum_{i=1}^{p+1}(-1)^{i-1}\xi_{i}\big((\xi_1\wedge\cdots\widehat{\xi_{i}}\cdots\wedge\xi_{p+1})(\omega)\big)\\
+&\sum_{1\leq i<j\leq p+1}(-1)^{i+j}([\xi_i,\xi_j]
\wedge\xi_1\wedge\cdots\widehat{\xi_i}\cdots\widehat{\xi_j}\cdots\wedge\xi_{p+1})(\omega).
\end{align*}
\end{lem}

\begin{proof} Assume that $\omega=a_1\de\!a_2\wedge\cdots \wedge \de\!a_{p+1} \in
\Omega^p(R).$
Then, by definition,
\begin{align*}
&\sum_{i=1}^{p+1}(-1)^{i-1}\xi_{i}\big((\xi_1\wedge\cdots\widehat{\xi_{i}}\cdots\wedge\xi_{p+1})(\omega)\big)\\
=&\sum_{i=1}^{p+1}(-1)^{i-1} \xi_{i}(a_1) \big((\xi_1\wedge\cdots\widehat{\xi_{i}}\cdots\wedge\xi_{p+1})(a_2\wedge \cdots\wedge a_{p+1})\big)\\
& + \sum_{i=1}^{p+1}(-1)^{i-1} a_1
\xi_{i}\big((\xi_1\wedge\cdots\widehat{\xi_{i}}\cdots\wedge\xi_{p+1}) (a_2\wedge \cdots\wedge a_{p+1})\big)\\
=&(\xi_1\wedge\cdots\wedge\xi_{p+1})(\de\!a_1\wedge\de\!a_2\wedge\cdots \wedge \de\!a_{p+1})\\
& +  \sum_{1\leq j < i\leq p+1}(-1)^{i-1}a_1
\big((\xi_1\wedge\cdots\wedge\xi_i\xi_j\wedge\cdots \widehat{\xi_i}
\cdots \wedge\xi_{p+1})
 (a_2\wedge \cdots\wedge a_{p+1})\big)\\
& +  \sum_{1\leq i < j\leq p+1}(-1)^{i-1}a_1 \big((\xi_1\wedge\cdots
\widehat{\xi_i}\cdots\wedge\xi_i\xi_j\wedge\cdots\wedge\xi_{p+1})
 (a_2\wedge \cdots\wedge a_{p+1})\big)\\
=& (\xi_1\wedge\cdots\wedge\xi_{p+1})(\de \omega) - \sum_{1\leq i<j\leq
p+1}(-1)^{i+j}([\xi_i,\xi_j]\wedge\xi_1\wedge\cdots\widehat{\xi_i}\cdots\widehat{\xi_j}\cdots\wedge\xi_{p+1})(\omega).
\end{align*}
\end{proof}
\begin{lem}\label{main-lemma}
If $P\in \mathfrak{X}^p(R),Q\in \mathfrak{X}^q(R)$ and $\omega\in
\Omega^{p+q-1}(R)$, then
\begin{equation}\label{con-map-SN}
\iota_{[P,Q]}(\omega)=(-1)^{(p-1)(q-1)}\iota_{P}(\de \iota
_Q(\omega))-\iota_Q(\de \iota_P(\omega))+(-1)^p\iota_{P\wedge Q}(\de
\omega).
\end{equation}
\end{lem}
\begin{proof}
Without loss of generality, suppose
$P=\xi_1\wedge\cdots\wedge\xi_p$,
$Q=\eta_1\wedge\cdots\wedge\eta_q$. Then
\begin{align*}
&(-1)^p\iota_{P\wedge Q}(\de \omega)\\
=&(-1)^p(\xi_1\wedge\cdots\wedge\xi_p\wedge\eta_1\wedge\cdots\wedge\eta_q)(\de \omega)\\
=&\sum_{i=1}^{p}(-1)^{p+i-1}\xi_i\big((\xi_1\wedge\cdots\widehat{\xi_i}\cdots\wedge\xi_p\wedge\eta_1\wedge
\cdots\wedge\eta_q)(\omega)\big)\qquad\qquad\mbox{(by Lemma  \ref{a})}\\
&+\sum_{i=1}^{q}(-1)^{i-1}\eta_i\big((\xi_1\wedge\cdots\wedge\xi_p\wedge\eta_1\wedge
\cdots\widehat{\eta_i}\cdots\wedge\eta_q)(\omega)\big)\\
&+\sum_{1\leq i<j\leq p}(-1)^{p+i+j}([\xi_i,\xi_j]\wedge\xi_1\wedge\cdots\widehat{\xi_i}
\cdots\widehat{\xi_j}\cdots\wedge\xi_p\wedge\eta_1\wedge\cdots\wedge\eta_q)(\omega)\\
&+\sum_{1\leq i<j\leq
q}(-1)^{p+i+j}([\eta_i,\eta_j]\wedge\xi_1\wedge
\cdots\wedge\xi_p\wedge\eta_1\wedge\cdots\widehat{\eta_i}\cdots\widehat{\eta_j}\cdots\wedge\eta_q)(\omega)\\
&+\sum_{i=1}^{p}\sum_{j=1}^{q}(-1)^{i+j}([\xi_i,\eta_j]\wedge\xi_1\wedge
\cdots\widehat{\xi_i}\cdots\wedge\xi_p\wedge\eta_1\wedge\cdots\widehat{\eta_j}\cdots\wedge\eta_q)(\omega)\\
=&-(-1)^{(p-1)(q-1)}\iota_{P}(\de \iota
_Q(\omega))+\iota_Q(\de \iota_P(\omega))\\
&+\sum_{i=1}^{p}\sum_{j=1}^{q}(-1)^{i+j}([\xi_i,\eta_j]\wedge\xi_1\wedge \cdots\widehat{\xi_i}\cdots\wedge\xi_p\wedge\eta_1\wedge\cdots\widehat{\eta_j}\cdots\wedge\eta_q)(\omega)\\
=&-(-1)^{(p-1)(q-1)}\iota_{P}(\de \iota
_Q(\omega))+\iota_Q(\de \iota_P(\omega))+\iota_{[P,Q]}(\omega) \quad\qquad\mbox{(see Example \ref{general-G-stru-2})}.
\end{align*}
\end{proof}



\begin{lem}\label{key-lemma} If $P\in \mathfrak{X}^p(R)$ and $\omega\in \Omega^{p-1}(R)$, then
  $$\iota_\omega(\Delta P)=\Delta(\iota_\omega(P))+\iota_{\de\!\omega}(P).$$
\end{lem}
\begin{proof}
By Theorem  \ref{dual iso}, $P=\iota_\alpha(\vol^*)$ for some $\alpha\in \Omega^{n-p}(R)$.
\begin{align*}
  \Delta(\iota_\omega(P))
  &=\Delta(\iota_{\omega\wedge\alpha}(\vol^*))&\qquad \mbox{ (by Definition  \ref{ctr2})}\\
  &=-\iota_{\de (\omega\wedge\alpha)}(\vol^*) &\qquad \mbox{ (by \eqref{def-of-bv-operator})}\\
  &=-\iota_{\de\!\omega\wedge\alpha}(\vol^*)+(-1)^{p}\iota_{\omega\wedge\de\!\alpha}(\vol^*)\\
  &=-\iota_{\de\!\omega}(P)+\iota_\omega(\Delta P).
\end{align*}
\end{proof}

For the geometric version of Lemma \ref{key-lemma}, see  \cite[Lemma 3.5]{Xu}.


\begin{thm}\label{main thm-BV}
Let $R$ be a smooth Poisson algebra with trivial canonical bundle. The triple $(\mathfrak{X}^*(R),\wedge,\Delta)$ is a BV algebra with the BV operator $\Delta$ given in \eqref{def-of-bv-operator}.
\end{thm}

\begin{proof} It suffices to show that for any $P\in \mathfrak{X}^p(R)$, $Q\in \mathfrak{X}^q(R)$,
$$[P,Q]=(-1)^p(\Delta (P\wedge Q)-\Delta P\wedge Q-(-1)^p P\wedge\Delta Q).$$
For any $\omega\in \Omega^{p+q-1}(R)$, by Lemma  \ref{key-lemma},
\begin{align*}
\iota_\omega(\Delta(P\wedge Q))=\Delta(\iota_\omega(P\wedge Q))+\iota_{\de\!\omega}(P\wedge Q).
\end{align*}
If we take  $\alpha=\iota_Q(\omega)$ and $\beta=\iota_P(\omega)$, then by Lemma  \ref{key-lemma} again,
\begin{align*}
\iota_\omega(\Delta P\wedge Q)&=(-1)^{(p-1)q}\iota_\omega(Q\wedge\Delta P)\\
&=(-1)^{(p-1)q}\iota_\alpha(\Delta P)\\
&=(-1)^{(p-1)q}\Delta(\iota_\alpha(P))+(-1)^{(p-1)q}\iota_{\de\!\alpha}(P).
\end{align*}
Similarly,
\begin{align*}
\iota_\omega( P\wedge\Delta Q)=\Delta(\iota_\beta(Q))+\iota_{\de\!\beta}(Q).
\end{align*}
Hence,
\begin{align*}
&\iota_\omega[(-1)^p\Delta (P\wedge Q)-(-1)^p\Delta P\wedge Q- P\wedge\Delta Q)]\\
=&(-1)^p\Delta(\iota_\omega(P\wedge Q))-(-1)^{p+(p-1)q}\Delta(\iota_\alpha(P))-\Delta(\iota_\beta(Q))\\
&+(-1)^p\iota_{\de\!\omega}(P\wedge Q)-(-1)^{p+(p-1)q}\iota_{\de\!\alpha}(P)-\iota_{\de\!\beta}(Q)\\
=&(-1)^p\iota_{\de\!\omega}(P\wedge Q)-(-1)^{p+(p-1)q}\iota_{\de\!\alpha}(P)-\iota_{\de\!\beta}(Q) \, \quad\mbox{ (by Lemma  \ref{con-map-wedge})} \\
=&\iota_\omega([P,Q]). \qquad\qquad\qquad\qquad\qquad\qquad \qquad\qquad\qquad \mbox{ (by Lemma  \ref{main-lemma})}
\end{align*}
The proof is finished.
\end{proof}

There are some interesting results induced by the BV algebra structure on $\mathfrak{X}^*(R)$. By  \cite[Lemma 5]{LWW1}), for any
 $P\in \mathfrak{X}^p(R)$ and $Q\in \mathfrak{X}^q(R)$,
 $$\Delta([P,Q] )=[\Delta(P),Q] -(-1)^p
[P,\Delta(Q)] .$$
Especially, for $P,Q\in \mathfrak{X}^1(R)$,  $$\Delta([P,Q])=P(\Delta(Q))-Q(\Delta(P)).$$

Furthermore,  from \cite[Proposition 5]{LWW1}), for any $P\in \mathfrak{X}^p(R)$, \begin{equation}\label{Detdet+detDet}   (\Delta \delta +  \delta \Delta)(P)=[\Delta(\pi), P] .
\end{equation}

\begin{prop}\label{phi_vol=Delta(pi)}
Let $R$ be a smooth Poisson algebra with trivial canonical bundle,  $\phi_{\vol}$ be its modular derivation, and  $\Delta$ be the operator given in  \eqref{def-of-bv-operator}. Then
\begin{enumerate}
\item
$\phi_{\vol}=\Delta(\pi)$;
 \item   for any $P\in \mathfrak{X}^p(R)$, $  (\Delta \delta +  \delta \Delta)(P)=[\phi_{\vol}, P].$
\end{enumerate}
\end{prop}
\begin{proof}
(1) By the definition of modular derivation, for any $a\in R$, $\phi_{\vol}(a)=\iota_{\de\!\iota_{H_a}(\vol)}\vol^*.$
It follows from the definition of $\Delta$ that $\phi_{\vol}(a)=-\Delta(H_a).$

By \eqref{Detdet+detDet}, $(\Delta \delta +  \delta \Delta)(a)=[\Delta(\pi), a].$  Note that $\delta(a)=-H_a$ and $\Delta(a)=0$, so $-\Delta (H_a)=[\Delta(\pi), a]=\Delta(\pi)(a)$.
Hence $\phi_{\vol}(a)=-\Delta(H_a)=\Delta(\pi)(a)$ for any $a\in R$, that is $\phi_{\vol}=\Delta(\pi)$.

(2) It follows from \eqref{Detdet+detDet}.
 \end{proof}

\begin{cor} With the assumption and notation as in Proposition \ref{phi_vol=Delta(pi)}.  Then
\begin{enumerate}
\item
for any Poisson derivation $P\in \mathfrak{X}^1(R)$, $\Delta(P)$ is a Casimir element if and only if $[\phi_{\vol}, P]=0;$
\item for any  Casimir element $a\in R,$ $\phi_{\vol}(a)=0.$
\end{enumerate}
\end{cor}
\begin{proof}
(1) Since $P$ is a Poisson derivation, $\delta(P)=0.$ Hence, by Proposition \ref{phi_vol=Delta(pi)},
\begin{align*}
\Delta(P) \text{is a Casimir element}
\Leftrightarrow \delta(\Delta(P))=0
\Leftrightarrow (\Delta \delta + \delta \Delta)(P)=0
\Leftrightarrow [\phi_{\vol}, P] =0.
\end{align*}

(2) For any $a\in R,$ $\Delta(a)=0.$ Since $a$ is a Casimir element, $\delta(a)=0.$ Then, by Proposition \ref{phi_vol=Delta(pi)}, $\phi_{\vol}(a)=[\phi_{\vol}, a]=(\Delta \delta +  \delta \Delta)(a)=0.$
\end{proof}

\subsection{BV operator on $\mathfrak{X}^*(R)$}
In this subsection we will describe the BV operator $\Delta$ precisely given in the previous subsection.
Notations are as in Section \ref{mod-der-in-dual-basis}.

\begin{thm} \label{Description-BV-Ooerator}
Let $R$ be a smooth Poisson algebra with trivial canonical bundle, and $\Delta$ be the operator given in
\eqref{def-of-bv-operator}. Then  for each $P \in\mathfrak{X}^{p}(R)$,
\begin{align*}
\Delta(P)(a_1\wedge a_2\wedge \cdots\wedge a_{p-1})
=&(-1)^p\sum_{1\leq l\leq r}(\de\!x_l)^*\big(P(a_1\wedge a_2\wedge \cdots\wedge a_{p-1}\wedge x_l)\big)\\
&+(-1)^p\sum_{I\in S}P(a_1\wedge a_2\wedge \cdots\wedge a_{p-1}\wedge  a_I)b_I.
\end{align*}
\end{thm}
\begin{proof}
Since $\vol=\sum_{I\in S}a_I\de\!x_I$,
$
\iota_{P}(\vol)=\sum_{I\in S}a_I\iota_{P}(\de\!x_I)
$
and
\begin{align*}
&\de\!\iota_{P}(\vol)\\
=&\sum_{I\in S}\de\!a_I\wedge \iota_{P}(\de\!x_I)+\sum_{I\in S}a_I\de\!\iota_{P}(\de\!x_I)\\
\stackrel{(a)}{=}&\sum_{I\in S}(-1)^{n-p}\big(\iota_P(\de\!x_I\wedge\de\!a_I)-(-1)^{n-p+1}\iota_{\iota_{\de\!a_I}P}(\de\!x_I)\big)\\
&+\sum_{I\in S}\sum_{\sigma\in S_{p,n-p}}\sgn(\sigma)a_I\de\!P(x_{I_{\sigma(1)}}\wedge \cdots\wedge x_{I_{\sigma(p)}})
\wedge \de\!x_{I_{\sigma(p+1)}}\wedge \de\!x_{I_{\sigma(p+2)}}\wedge \cdots \wedge \de\!x_{I_{\sigma(n)}}\\
=&\sum_{I\in S}\iota_{\iota_{\de\!a_I}P}(\de\!x_{I})\\
&+\sum_{I\in S}\sum_{\sigma\in S_{p,n-p}}\sgn(\sigma)a_I\sum_{1\leq l\leq r}(\de\!x_{l})^*(P(x_{I_{\sigma(1)}}\wedge \cdots\wedge x_{I_{\sigma(p)}}))\\
&\quad \quad
\de\!x_l\wedge \de\!x_{I_{\sigma(p+1)}}\wedge \de\!x_{I_{\sigma(p+2)}}\wedge \cdots \wedge \de\!x_{I_{\sigma(n)}}\\
=&\sum_{I\in S}\iota_{\iota_{\de\!a_I}P}(\de\!x_{I}) \qquad \qquad \qquad \qquad \qquad \qquad \qquad \qquad \qquad \qquad \quad (\mbox{denoted by }U)\\
&+\sum_{I\in S}\sum_{\sigma\in S_{p,n-p}}\sgn(\sigma)a_I\sum_{1\leq l\leq n}(\de\!x_{\sigma(l)})^*(P(x_{I_{\sigma(1)}}\wedge \cdots\wedge x_{I_{\sigma(p)}}))\\
&\quad \quad
\de\!x_{\sigma(l)}\wedge \de\!x_{I_{\sigma(p+1)}}\wedge \de\!x_{I_{\sigma(p+2)}}\wedge \cdots \wedge \de\!x_{I_{\sigma(n)}} \qquad \qquad \qquad \quad (\mbox{denoted by }V)\\
&+\sum_{I\in S}\sum_{\sigma\in S_{p,n-p}}\sgn(\sigma)a_I\sum_{l \notin I}(\de\!x_{l})^*(P(x_{I_{\sigma(1)}}\wedge \cdots\wedge x_{I_{\sigma(p)}}))\\
&\quad \quad
\de\!x_l\wedge \de\!x_{I_{\sigma(p+1)}}\wedge \de\!x_{I_{\sigma(p+2)}}\wedge \cdots \wedge \de\!x_{I_{\sigma(n)}}\qquad \qquad \qquad \qquad (\mbox{denoted by }W)
\end{align*}
where $(a)$ holds by Proposition \ref{ctr-equality}, i.e. the equation \eqref{ctr-equality-no}:
\begin{equation*}
\iota_F(\omega \wedge \de\! a)= \iota_{F}(\omega)\wedge \de\! a +
 (-1)^{q-p+1}\iota_{\iota_{\de\!a}(F)}(\omega)
\end{equation*}
for any $a \in R$, $F\in \mathfrak{X}^{p}(R)$ and $\omega\in \Omega^q(R)$. 

Next we calculate the terms $V$ and $W$ respectively.
\begin{align*}
V=&\sum_{I\in S}\sum_{\sigma\in S_{p,n-p}}\sgn(\sigma)a_I\sum_{1\leq l\leq p}(\de\!x_{\sigma(l)})^*(P(x_{I_{\sigma(1)}}\wedge \cdots\wedge x_{I_{\sigma(p)}}))\\
&\quad \quad
\de\!x_{\sigma(l)}\wedge \de\!x_{I_{\sigma(p+1)}}\wedge \de\!x_{I_{\sigma(p+2)}}\wedge \cdots \wedge \de\!x_{I_{\sigma(n)}}\\
\stackrel{(b)}{=}&\sum_{I\in S}\sum_{\tau\in S_{p-1,n-p+1}}\sgn(\tau)a_I\sum_{p\leq j\leq n}(\de\!x_{\tau(j)})^*(P(x_{I_{\tau(1)}}\wedge \cdots\wedge x_{I_{\tau(p-1)}}\wedge x_{I_{\tau(j)}}))\\
&\quad \quad
 \de\!x_{I_{\tau(p)}}\wedge \de\!x_{I_{\tau(p+1)}}\wedge \cdots \wedge \de\!x_{I_{\tau(n)}}\\
=&\sum_{I\in S}\sum_{\tau\in S_{p-1,n-p+1}}\sgn(\tau)a_I\sum_{1\leq j\leq n}(\de\!x_{\tau(j)})^*(P(x_{I_{\tau(1)}}\wedge \cdots\wedge x_{I_{\tau(p-1)}}\wedge x_{I_{\tau(j)}}))\\
&\quad \quad
 \de\!x_{I_{\tau(p)}}\wedge \de\!x_{I_{\tau(p+1)}}\wedge \cdots \wedge \de\!x_{I_{\tau(n)}}
\end{align*}
where $(b)$ holds by the one-to-one
correspondence $$\{(\sigma,l)\mid \sigma\in S_{p,n-p}, 1\leq l \leq p\} \to \{(\tau,j)\mid \tau\in S_{p-1,n-p+1}, p\leq j
\leq n \}.$$
\begin{align*}
W=&\sum_{I\in S}\sum_{\sigma\in S_{p,n-p}}\sgn(\sigma)a_I\sum_{l \notin I}(\de\!x_{l})^*(P(x_{I_{\sigma(1)}}\wedge \cdots\wedge x_{I_{\sigma(p)}}))\\
&\quad \quad
\de\!x_l\wedge \de\!x_{I_{\sigma(p+1)}}\wedge \de\!x_{I_{\sigma(p+2)}}\wedge \cdots \wedge \de\!x_{I_{\sigma(n)}}\\
\stackrel{(c)}{=}&\sum_{I\in S}\sum_{\sigma\in S_{p,n-p}}\sum_{l \notin I}\sum_{1\leq j\leq n}\sgn(\sigma)\vol(\de\!x_{I_{j}\to l}^* ) (\de\!x_{I_{j}})^*(P(x_{I_{\sigma(1)}}\wedge \cdots\wedge x_{I_{\sigma(p)}}))\\
&\qquad
\de\!x_l\wedge \de\!x_{I_{\sigma(p+1)}}\wedge \de\!x_{I_{\sigma(p+2)}}\wedge \cdots \wedge \de\!x_{I_{\sigma(n)}}\\
=&\sum_{I\in S}\sum_{\sigma\in S_{p,n-p}}\sum_{l \notin I}\sum_{1\leq j\leq p}\sgn(\sigma)\vol(\de\!x_{I_{\sigma(j)\to l}}^*)  (\de\!x_{I_{\sigma(j)}})^*(P(x_{I_{\sigma(1)}}\wedge \cdots\wedge x_{I_{\sigma(p)}}))\\
&\qquad
\de\!x_l\wedge \de\!x_{I_{\sigma(p+1)}}\wedge \de\!x_{I_{\sigma(p+2)}}\wedge \cdots \wedge \de\!x_{I_{\sigma(n)}} \qquad \qquad \qquad (\mbox{denoted by }W_1) \\
+&\sum_{I\in S}\sum_{\sigma\in S_{p,n-p}}\sum_{l \notin I}\sum_{p+1\leq j\leq n}\sgn(\sigma)\vol(\de\!x_{I_{\sigma(j)\to l}}^*) (\de\!x_{I_{\sigma(j)}})^*(P(x_{I_{\sigma(1)}}\wedge \cdots\wedge x_{I_{\sigma(p)}}))\\
&\qquad
\de\!x_l\wedge \de\!x_{I_{\sigma(p+1)}}\wedge \de\!x_{I_{\sigma(p+2)}}\wedge \cdots \wedge \de\!x_{I_{\sigma(n)}}  \qquad \qquad \qquad (\mbox{denoted by }W_2)
\end{align*}
where $(c)$ holds by Lemma \ref{change one}.
In order to compute $W_1$, consider the one-to-one
correspondence from the set $\{(I,l,\sigma,j)\mid I\in S, l\notin I,\sigma\in S_{p,n-p},1\leq j
\leq p\}$ to $\{(I,l,\tau,j)\mid I\in S, l\notin I,\tau\in S_{p-1,n-p+1},p\leq j
\leq n \}$:
$$ (I,l,\sigma, 1\leq j \leq   p) \mapsto (I', l',\tau,p\leq j' \leq n),$$
where $I'=(I \backslash \{I_{\sigma(j)}\})\cup\{l\}= \{I_1, \cdots, \widehat{I_{\sigma(j)}}, \cdots, I_n, l\}$, $l'=I_{\sigma(j)}$, $$\{I'_{\tau(1)}, I'_{\tau(2)}, \cdots, I'_{\tau(p-1)}\}=\{I_{\sigma(1)}, \cdots, \widehat{I_{\sigma(j)}}, \cdots, I_{\sigma(p)}\},$$ $$\{I'_{\tau(p)}, I'_{\tau(p+1)}, \cdots, I'_{\tau(n)}\}=\{I_{\sigma(p+1)},  I_{\sigma(p+2)}, \cdots, I_{\sigma(n)}, l\},$$ and
$j'$ is the unique number satisfying $I'_{\tau (j')}= l$.
Then
\begin{align*}
W_1
=&\sum_{I\in S}\sum_{\sigma\in S_{p,n-p}}\sum_{l \notin I}\sum_{1\leq j\leq p}\sgn(\sigma)\vol(\de\!x_{I_{\sigma(j)}\to l}^*)(\de\!x_{I_{\sigma(j)}})^*(P(x_{I_{\sigma(1)}}\wedge \cdots\wedge x_{I_{\sigma(p)}}))\\
&\qquad
\de\!x_l\wedge \de\!x_{I_{\sigma(p+1)}}\wedge \de\!x_{I_{\sigma(p+2)}}\wedge \cdots \wedge \de\!x_{I_{\sigma(n)}}\\
=&\sum_{I'\in S}\sum_{\tau\in S_{p-1,n-p+1}}\sum_{l' \notin I'}\sum_{p\leq j'\leq n}\sgn(\tau)a_{I'} (\de\!x_{l'})^*(P(x_{I'_{\tau(1)}}\wedge \cdots\wedge x_{I'_{\tau(p-1)}}\wedge x_{l'}))\\
&\qquad
\de\!x_{I'_{\tau(p)}}\wedge \de\!x_{I'_{\tau(p+1)}}\wedge \cdots \wedge \de\!x_{I'_{\tau(n)}}\\
=&(n-p+1)\sum_{I\in S}\sum_{\tau\in S_{p-1,n-p+1}}\sum_{l \notin I}\sgn(\tau)a_{I} (\de\!x_{l})^*(P(x_{I_{\tau(1)}}\wedge \cdots\wedge x_{I_{\tau(p-1)}}\wedge x_{l}))\\
&\qquad
\de\!x_{I_{\tau(p)}}\wedge \de\!x_{I_{\tau(p+1)}}\wedge \cdots \wedge \de\!x_{I_{\tau(n)}}.
\end{align*}
On the other hand,
\begin{align*}
W_2=&\sum_{I\in S}\sum_{\sigma\in S_{p,n-p}}\sum_{l \notin I}\sum_{p+1\leq j\leq n}\sgn(\sigma)\vol(\de\!x_{I_{\sigma(j)}\to l}^*) (\de\!x_{I_{\sigma(j)}})^*(P(x_{I_{\sigma(1)}}\wedge \cdots\wedge x_{I_{\sigma(p)}}))\\
&\qquad
\de\!x_l\wedge \de\!x_{I_{\sigma(p+1)}}\wedge \de\!x_{I_{\sigma(p+2)}}\wedge \cdots \wedge \de\!x_{I_{\sigma(n)}}\\
=&\sum_{I'\in S}\sum_{\sigma'\in S_{p,n-p}}\sum_{l' \notin I'}\sum_{p+1\leq j'\leq n}(-1)\sgn(\sigma')\vol(\de\!x_{I'}) (\de\!x_{l'})^*(P(x_{I'_{\sigma'(1)}}\wedge \cdots\wedge x_{I'_{\sigma'(p)}}))\\
&\qquad
\de\!x_{l'}\wedge \de\!x_{I'_{\sigma'(p+1)}}\wedge \de\!x_{I'_{\sigma'(p+2)}}\wedge \cdots \wedge \de\!x_{I'_{\sigma'(n)}}\\
=&-(n-p)\sum_{I\in S}\sum_{\sigma\in S_{p,n-p}}\sum_{l \notin I}\sgn(\sigma)a_I (\de\!x_{l})^*(P(x_{I_{\sigma(1)}}\wedge \cdots\wedge x_{I_{\sigma(p)}}))\\
&\qquad
\de\!x_{l}\wedge \de\!x_{I_{\sigma(p+1)}}\wedge \de\!x_{I_{\sigma(p+2)}}\wedge \cdots \wedge \de\!x_{I_{\sigma(n)}}\\
=&-(n-p)W.
\end{align*}
It follows that
$W=W_1+W_2=W_1-(n-p)W$ and $W_1=(n-p+1)W$.
Hence
\begin{align*}
W=&\sum_{I\in S}\sum_{\tau\in S_{p-1,n-p+1}}\sum_{l \notin I}\sgn(\tau)a_{I} (\de\!x_{l})^*(P(x_{I_{\tau(1)}}\wedge \cdots\wedge x_{I_{\tau(p-1)}}\wedge x_{l}))\\
&\qquad
\de\!x_{I_{\tau(p)}}\wedge \de\!x_{I_{\tau(p+1)}}\wedge \cdots \wedge \de\!x_{I_{\tau(n)}},
\end{align*}
\begin{align*}
V+W=&\sum_{I\in S}\sum_{\tau\in S_{p-1,n-p+1}}\sum_{1\leq l \leq r}\sgn(\tau)a_{I} (\de\!x_{l})^*(P(x_{I_{\tau(1)}}\wedge \cdots\wedge x_{I_{\tau(p-1)}}\wedge x_{l}))\\
&\qquad
\de\!x_{I_{\tau(p)}}\wedge \de\!x_{I_{\tau(p+1)}}\wedge \cdots \wedge \de\!x_{I_{\tau(n)}}.
\end{align*}
Note that $\sum_{1\leq l \leq r}(\de\!x_{l})^*[P(-\wedge \cdots\wedge -\wedge x_{l})]$ is a $(p-1)$-fold multi-derivation, and
\begin{align*}
V+W=&\sum_{I\in S}\iota_{\big(\sum_{1\leq l \leq r}(\de\!x_{l})^*[P(-\wedge \cdots\wedge -\wedge x_{l})]\big)}(a_I\de\!x_I)\\
=&\iota_{\big(\sum_{1\leq l \leq r}(\de\!x_{l})^*[P(-\wedge \cdots\wedge -\wedge x_{l})]\big)}(\vol).
\end{align*}
Then \begin{align*}
&\de\!\iota_{P}(\vol)\\
=&\sum_{I\in S}\iota_{\iota_{\de\!a_I}P}(\de\!x_{I})+\iota_{\big(\sum_{1\leq l \leq r}(\de\!x_{l})^*[P(-\wedge \cdots\wedge -\wedge x_{l})]\big)}(\vol)\\
=&\sum_{I\in S}b_I\iota_{\iota_{\de\!a_I}P}(\vol)+\iota_{\big(\sum_{1\leq l \leq r}(\de\!x_{l})^*[P(-\wedge \cdots\wedge -\wedge x_{l})]\big)}(\vol).
\end{align*}

By the definition of the BV operator $\Delta$,
\begin{align*}
\Delta(P)(a_1\wedge a_2\wedge \cdots\wedge a_{p-1})
=&(-1)^p\sum_{1\leq l\leq r}(\de\!x_l)^*\big(P(a_1\wedge a_2\wedge \cdots\wedge a_{p-1}\wedge x_l)\big)\\
&+(-1)^p\sum_{I\in S}P(a_1\wedge a_2\wedge \cdots\wedge a_{p-1}\wedge  a_I)b_I.
\end{align*}
\end{proof}

\begin{rmk}\label{phi=DetPi}
For $P\in \mathfrak{X}^1(R)$, $\Delta P=-\sum_{1\leq l\leq r}(\de\!x_l)^*\big(P(x_l)\big)-\sum_{I\in S}P(a_I)b_I.$
For $P\in \mathfrak{X}^2(R)$, $(\Delta P) (a)=\sum_{1\leq l\leq r}(\de\!x_l)^*\big(P(a\wedge x_l)\big)+\sum_{I\in S}P(a\wedge a_I)b_I.$ Especially, $\phi_{\vol}=\Delta (\pi) $ by Theorem \ref{modular der of smooth}, which is consistent with Proposition \ref{phi_vol=Delta(pi)}.
\end{rmk}

\begin{exam}
Let $R={\mathbbm{k}}[x_1, x_2, \cdots, x_n]$ be a polynomial
algebra. The triple $(\mathfrak{X}^*(R),\wedge,\Delta)$ is a BV algebra, where
  $$  \Delta(P)(a_1\wedge a_2\wedge \cdots\wedge a_{p-1})=(-1)^p\sum_{1\leq l\leq n}\frac{\partial}{\partial x_{l}}\big(P(a_1\wedge a_2\wedge \cdots\wedge a_{p-1}\wedge x_l)\big)$$
  for any $P \in\mathfrak{X}^{p}(R)$ and $a_1, a_2, \cdots, a_{p-1}\in R.$
If $P=a\frac{\partial}{\partial x_{i_1}}\wedge\frac{\partial}{\partial x_{i_2}}\wedge\cdots\wedge
  \frac{\partial}{\partial x_{i_p}}$,
then
    $$\Delta(P)=\sum_{j=1}^{p}(-1)^{j}
  \frac{\partial a}{\partial x_{i_j}}\frac{\partial}{\partial x_{i_1}}\wedge
  \frac{\partial}{\partial x_{i_2}}\wedge\cdots
   \widehat{\frac{\partial}{\partial x_{i_j}}}\cdots
  \wedge\frac{\partial}{\partial x_{i_p}}.$$
\end{exam}

\subsection{BV structure on Poisson cohomology of unimodular Poisson algebra}

Now consider the smooth Poisson algebra $R$ with trivial canonical bundle. We will investigate the BV structures on its Poisson cohomology in the case that the Poisson structure is unimodular.
\begin{lem} \label{partial-de+dp=0}
The following diagram is anti-commutative
$$
\xymatrix{
\cdots \ar[r]& \Omega^{q}(R)\ar[d]_{\partial}\ar[r]^{\de}& \Omega^{q+1}(R)\ar[d]_{\partial}\ar[r]&\cdots\\
\cdots \ar[r] & \Omega^{q-1}(R)\ar[r]^{\de}&
\Omega^{q}(R)\ar[r]&\cdots }
$$
\end{lem}

\begin{proof} For any $a_0\de\!a_1\wedge\de\!a_2\wedge\cdots \wedge \de\!a_{q} \in \Omega^{q}(R)$, by definition,
\begin{align*}
&\de \partial (a_0\de\!a_1\wedge\de\!a_2\wedge\cdots \wedge \de\!a_{q})\\
=&\sum_{1 \leq i \leq q} (-1)^{i-1} \de \{a_0, a_i\} \wedge \de\!a_1\wedge \cdots \widehat{\de\!a_i}\cdots \wedge \de\!a_{q}\\
&+\sum_{1 \leq i , j \leq q} (-1)^{i+j} \de\!a_0\wedge \de \{a_i, a_j\} \wedge \de\!a_1\wedge \cdots \widehat{\de\!a_i}\cdots \widehat{\de\!a_j}\cdots\wedge \de\!a_{q}\\
=&- \partial (\de\!a_0 \wedge \de\!a_1\wedge\de\!a_2\wedge\cdots \wedge \de\!a_{q})\\
=&- \partial \de (a_0\de\!a_1\wedge\de\!a_2\wedge\cdots \wedge \de\!a_{q}).
\end{align*}

\end{proof}
\begin{thm}\label{BV for unim poly}
If $R$ is an unimodular Poisson algebra, then its Poisson cohomology $\PH^*(R)$ admits  a BV algebra structure induced from the one on $\mathfrak{X}^*(R)$ given in Theorem \ref{main thm-BV}.
\end{thm}
\begin{proof}
By Lemma \ref{wedge cocycle}, the product
$$\wedge : \,\PH^p(R)\times \PH^q(R)\to \PH^{p+q}(R), \, \overline{F}\times \overline{G}\mapsto \overline{F\wedge G}$$
is well-defined.

Note that $\partial \de+\de \partial=0$ by Lemma \ref{partial-de+dp=0}. The conclusion follows directly from Remark \ref{unimodular-com-dia}, Theorem \ref{main thm-BV} and the following (anti-)commutative diagram, which indicates that the operator $\Delta$ can also pass through the cohomology differential $\delta$ to its cohomology groups.
\[
\xymatrixrowsep{0.2 in}
\xymatrixcolsep{0.1 in}
\xymatrix{
&&&& &&&\cdot &&& \cdot  &&& \\
&&&\cdots \ar[rrr]&&&\mathfrak{X}^{p+1}(R)\ar@{.>}[ddd]^{\dag_R^{p+1}}\ar[rrr]^{\Delta}\ar[ur]&&&
\mathfrak{X}^p(R)\ar@{.>}[ddd]^{\dag_R^{p}}\ar[r]\ar[ur]& \cdots\\
&&\cdots \ar[rrr]&&&\mathfrak{X}^p(R)\ar[ddd]^{\dag_R^{p}}\ar[rrr]^{\Delta}\ar[ur]^{\delta}&&& \mathfrak{X}^{p-1}(R)\ar[ddd]^{\dag_R^{p-1}}\ar[r]\ar[ur]^{\delta}&\cdots
\\
&&&& &&&\cdot &&& \cdot  &&& \\
&&&\cdots \ar@{.>}[rrr]&&&\Omega^{n-p-1}(R)\ar@{.>}[rrr]^{\de}\ar@{.>}[ur]&&&
\Omega^{n-p}(R)\ar@{.>}[r] \ar@{.>}[ur]& \cdots\\
&&\cdots \ar[rrr]&&&\Omega^{n-p}(R)\ar[rrr]^{\de}\ar@{.>}[ur]^{\partial}&&& \Omega^{n-p+1}(R)\ar[r]\ar@{.>}[ur]^{\partial}&\cdots
}\]
\end{proof}
%

\section{BV structure for pseudo-unimodular Poisson algebras}

In this section, a notion of  pseudo-unimodular Poisson algebras is given, and  a BV operator is constructed on the Poisson cohomology for any pseudo-unimodular Poisson algebra.  First recall two useful results.

\begin{lem}\cite[Lemma 2]{LWW1}\label{dt square zero}
Let $\varpi\in \Omega^1(R)$ be a 1-form and $\de'=\varpi \wedge - $. 
Let $\de_t=\de -\de'$. Then $\de_t^2=0$ if and only if $\varpi$ is a de Rham 1-cocycle.
\end{lem}

\begin{prop}\cite[Proposition 3]{LWW1}\label{twist dRdiff and part}
 Let $R$ be a Poisson algebra with the Poisson structure $\pi$ and $\varpi\in \Omega^1(R)$ be a de Rham 1-cocycle.  Then
\begin{enumerate}
\item $\phi=\iota_{\varpi}\pi\in \mathfrak{X}^1(R)$ is a Poisson derivation.
\item
$\partial_t =[\iota_{\pi}, \de_t]$, where $\partial_t=\partial^{R_{\phi}}$, $\de_t=\de -\de'$ and  $\de'=\varpi \wedge - $.
\item
$ \partial_t \de_t+\de_t \partial_t=0.$

\end{enumerate}
\end{prop}

Pseudo-unimodular Poisson structure for smooth algebras with trivial canonical bundle is defined similarly as for Frobenius Poisson algebras \cite[Definition 10]{LWW1}.

\begin{defn}\label{def of pseudo-unimodular} Let $(R, \pi)$ be a smooth Poisson algebra with trivial canonical bundle. Then $R$ is said to be  {\bf pseudo-unimodular} if there exists a de Rham 1-cocycle $\varpi \in  \Omega^1(R)$ such that $\iota_{\varpi}\pi$ is the modular derivation of $R$.
\end{defn}

\begin{rmk}
If Poisson algebra $R$ is unimodular, i.e. its modular derivation is a log-Hamiltonian derivation $u^{-1}\{-,u\}$ for some invertible element $u \in R$, then  $u^{-1}\{-,u\}=\iota_{u^{-1}\de\!u}\pi$, and  $\de(u^{-1}\de\!u)=0$. So $R$ is pseudo-unimodular.
\end{rmk}

\begin{exam}
If the modular derivation of a Poisson algebra $R$ is a Hamiltonian derivation, say, $\{u,-\}$, then $\{u,-\}=-\iota_{\de\!u}\pi$. Hence,  $R$ is also pseudo-unimodular (see \cite[Example 2.5]{LWW} for an example).
\end{exam}

\begin{rmk}
For polynomial Poisson algebras, any de Rham 1-cocycle must be an exact 1-form. Hence a polynomial Poisson algebra is pseudo-unimodular if and only if its modular derivation is a Hamiltonian derivation.
\end{rmk}

In the following, let $R$ be a pseudo-unimodular  smooth Poisson algebra  with trivial canonical bundle. By Proposition  \ref{twist dRdiff and part}, we can twist the de Rham differential by the de Rham 1-cocycle such that the twisted de Rham differential $\de_t$ is anti-commutative with the twisted Poisson differential  $\partial_t$ (with respect to the modular derivation).
Then we define a twisted differential operator $\Delta_t$ on $\mathfrak{X}^*(R)$ induced by the twisted de Rham differential $\de_t$ :
\begin{equation}\label{def of Dt}
\xymatrix{
   \mathfrak{X}^{p}(R) \ar[d]_{\dag_R^{p}} \ar@{-->}[r]^{\Delta_t}
        &   \mathfrak{X}^{p-1}(R)      \\
  \Omega^{n-p}(R)    \ar[r]^{\de_t}
        &  \Omega^{n-p+1}(R)\ar[u]_{(\dag_R^{p-1})^{-1}}.
     }
\end{equation}
It is easy to see that $\Delta_t=\Delta-\Delta'$, where $\Delta'=(\dag_R^{*-1})^{-1} \de'  \dag_R^{*}$, and $\de'=\varpi \wedge - $. We claim that $\Delta'$ is a super-derivation.

Keeping notations as above, we have the following lemma for pseudo-unimodular Poisson algebras.

\begin{lem}\label{delta' and con}
  For any $P\in \mathfrak{X}^p(R)$ and $\omega\in \Omega^{p-1}(R)$, 
  $\iota_\omega(\Delta' (P))=\Delta'(\iota_\omega P ).$
\end{lem}
\begin{proof}
By Theorem \ref{dual iso}, $P=\iota_\alpha(\vol^*)$ for some $\alpha\in \Omega^{n-p}(R)$. Then
\begin{align*}
  \Delta'(\iota_\omega(P))&=\Delta'(\iota_\omega(\iota_\alpha(\vol^*)))\\
  &=\Delta'(\iota_{\omega\wedge\alpha}(\vol^*)) &\mbox{ (by Definition  \ref{ctr2})}\\
  &=-\iota_{\de' (\omega\wedge\alpha)}(\vol^*) &\mbox{ (by the definition of $\Delta'$)}\\
  &=(-1)^{p}\iota_{\omega\wedge\de' \alpha}(\vol^*) &\mbox{ (by the definition of $\de'$)}\\
  &=\iota_{\omega }((-1)^{p}\iota_{\de' \alpha}(\vol^*))\\
  &=\iota_{\omega }(\Delta'( P)) &\mbox{ (by the definition of $\Delta'$)}
\end{align*}
\end{proof}

\begin{prop}\label{delta' der}
The operator $\Delta'$ is a super-derivation on the graded algebra $\mathfrak{X}^{*}(R)$, that is, for any
 $P\in \mathfrak{X}^p(R)$ and $Q\in \mathfrak{X}^q(R)$,
$$\Delta' (P\wedge Q)=\Delta' (P)\wedge Q+(-1)^p P\wedge\Delta'( Q).$$
\end{prop}
\begin{proof}
  For any $\omega\in \Omega^{p+q-1}(R)$, by Lemma  \ref{delta' and con}
 $\iota_\omega(\Delta'(P\wedge Q))=\Delta'(\iota_\omega(P\wedge Q)).$

If we set $\alpha=\iota_Q(\omega)$, then  by Lemma  \ref{delta' and con} again,
\begin{align*}
\iota_\omega(\Delta' (P)\wedge Q)&=(-1)^{(p-1)q}\iota_\omega(Q\wedge\Delta'( P))\\
&=(-1)^{(p-1)q}\iota_\alpha(\Delta'( P))\\
&=(-1)^{(p-1)q}\Delta'(\iota_\alpha(P)).
\end{align*}
If we set $\beta=\iota_P(\omega)$, then
$\iota_\omega( P\wedge\Delta'( Q))=\Delta'(\iota_\beta(Q))$.
By Lemma  \ref{con-map-wedge},
$$\iota_\omega(P\wedge Q)=(-1)^{(p-1)q}\iota_\alpha(P)+(-1)^p\iota_\beta(Q).$$
Hence
\begin{align*}
 \iota_\omega(\Delta'(P\wedge Q))=&\Delta'(\iota_\omega(P\wedge Q))\\
 =&(-1)^{(p-1)q}\Delta'(\iota_\alpha(P))+(-1)^p\Delta'(\iota_\beta(Q))\\
 =&\iota_\omega(\Delta' (P)\wedge Q)+(-1)^p\iota_\omega( P\wedge\Delta'( Q))\\
 =&\iota_\omega\big(\Delta' (P)\wedge Q+(-1)^p P\wedge\Delta' (Q)\big).
\end{align*}
This ends the proof.
\end{proof}

Recall that Theorem  \ref{main thm-BV} shows  the triple $(\mathfrak{X}^*(R),\wedge,\Delta)$ is a BV algebra. For any
 $P\in \mathfrak{X}^p(R)$ and $Q\in \mathfrak{X}^q(R),$
$$[P,Q]=(-1)^p(\Delta (P\wedge Q)-\Delta P\wedge Q-(-1)^p P\wedge\Delta Q).$$
It follows from Proposition  \ref{delta' der} that
$$[P,Q]=(-1)^p(\Delta_t (P\wedge Q)-\Delta_t P\wedge Q-(-1)^p P\wedge\Delta_t Q),$$
where $\Delta_t=\Delta-\Delta'$.
Hence $(\mathfrak{X}^*(R),\wedge,\Delta_t)$ is  also  a BV algebra.

\begin{thm} \label{Pseudo-unimodular has BV}
If $(R, \pi)$ is a pseudo-unimodular Poisson algebra, then
$(\PH^*(R),\wedge,\Delta_t)$ is a BV algebra.
\end{thm}
\begin{proof}

Since the Poisson structure is pseudo-unimodular,  there exists a  de Rham 1-cocycle $\varpi\in \Omega^1(R)$ such that its modular derivation $\phi_{\vol}=\iota_{\varpi}\pi$.
Consider the following (anti-)commutative diagram:
\[
\xymatrixrowsep{0.2 in}
\xymatrixcolsep{0.1 in}
\xymatrix{
&&&& &&&\cdot &&& \cdot  &&& \\
&&&\cdots \ar[rrr]&&&\mathfrak{X}^{p+1}(R)\ar@{.>}[ddd]^{\dag_R^{p+1}}\ar[rrr]^{\Delta_t}\ar[ur]&&&
\mathfrak{X}^p(R)\ar@{.>}[ddd]^{\dag_R^{p}}\ar[r]\ar[ur]& \cdots\\
&&\cdots \ar[rrr]&&&\mathfrak{X}^p(R)\ar[ddd]^{\dag_R^{p}}\ar[rrr]^{\Delta_t}\ar[ur]^{\delta}&&& \mathfrak{X}^{p-1}(R)\ar[ddd]^{\dag_R^{p-1}}\ar[r]\ar[ur]^{\delta}&\cdots
\\
&&&& &&&\cdot &&& \cdot  &&& \\
&&&\cdots \ar@{.>}[rrr]&&&\Omega^{n-p-1}(R)\ar@{.>}[rrr]^{\de_t}\ar@{.>}[ur]&&&
\Omega^{n-p}(R)\ar@{.>}[r] \ar@{.>}[ur]& \cdots\\
&&\cdots \ar[rrr]&&&\Omega^{n-p}(R)\ar[rrr]^{\de_t}\ar@{.>}[ur]^{\partial_t}&&& \Omega^{n-p+1}(R)\ar[r]\ar@{.>}[ur]^{\partial_t}&\cdots
}\]
where $\de_t=\de -\de'$ with $\de'=\varpi \wedge - $, and $\partial_t=\partial_{R_t}$ is the Poisson differential with values in $R_t$ which is the twisted Poisson module twisted by the modular derivation  $\phi_{\vol}$.  
By Proposition  \ref{twist dRdiff and part}, $\partial_t \de_t+\de_t \partial_t=0$.
Then by Lemma  \ref{wedge cocycle} and Theorem  \ref{comm diag1},
the operator $\wedge$, $\Delta_t$ can pass through the cohomology differential $\delta$ to its cohomology groups $\PH^*(R)$. Hence the BV algebra $(\mathfrak{X}^*(R),\wedge,\Delta_t)$ induces a BV algebra structure on $\PH^*(R)$.
\end{proof}

\begin{rmk}\label{delta'-and-iota_w}
For any $P\in \mathfrak{X}^{p}(R)$, $\Delta'(P)=(-1)^{p}\iota_{\varpi}P$. In fact,  by the definition of $\Delta'$, the following diagram 
\begin{equation}
\xymatrix{
   \mathfrak{X}^{p}(R) \ar[d]_{\dag_R^{p}}  \ar@{-->}[r]^{\Delta'}
        &   \mathfrak{X}^{p-1}(R)      \\
  \Omega^{n-p}(R)    \ar[r]^{\varpi \wedge -}
        &  \Omega^{n-p+1}(R) \ar[u]_{(\dag_R^{p-1})^{-1}}
     }
\end{equation}
is commutative. Recall that $\dag_R^{p}(P)=(-1)^{\frac{p(p+1)}{2}}\iota_P(\vol)$, and $$(\dag_R^{p-1})^{-1} (\varpi \wedge (\dag_R^{p}(P))) =(-1)^{\frac{(p-1)p}{2}}\iota_{\varpi \wedge (\dag_R^{p}(P))}\vol^*=(-1)^p\iota_{\varpi \wedge (\iota_P(\vol))}\vol^*.$$
For any $\eta \in \Omega^{p-1}(R)$,
\begin{align*}
\iota_{\varpi \wedge (\iota_P(\vol))}\vol^*(\eta)=&\vol^*\big(\eta \wedge \varpi \wedge (\iota_P\vol)\big)\\
=&[\iota_{\iota_P(\vol)}\vol^*](\eta \wedge \varpi)\\
=&P(\eta \wedge \varpi) &\mbox{ (by  Theorem \ref{dual iso})}\\
=&(\iota_{\varpi}P)(\eta),
\end{align*}
i.e., $\Delta'(P)=(-1)^{p}\iota_{\varpi}P$.
\end{rmk}

\begin{prop}\label{main-prop-1}
 For a smooth algebra $R$ with $\Omega^n(R)\cong R$, any BV operator (generating the Schouten-Nijenhuis  bracket) on $\mathfrak{X}^*(R)$  has the form $\Delta-(-1)^{|\,\,|}\iota_{\varpi}$, where $\Delta$ is the operator defined in \eqref{def-of-bv-operator}, $|\,\,|$ denotes the degree of homogeneous elements, and $\varpi\in \Omega^1(R)$ with $\de(\varpi)=0$.
\end{prop}
\begin{proof}
By Theorem \ref{main thm-BV}, $\Delta$ is a BV operator on $\mathfrak{X}^*(R)$ generating the Schouten-Nijenhuis  bracket. For any other BV operator $\Delta_t$, by the definition of BV operator \eqref{difference}, $\Delta-\Delta_t$ is a super-derivation.
Then $(\Delta -\Delta_t)\mid_{\mathfrak{X}^1(R)}$ is an $R$-module morphism. Thus $$(\Delta -\Delta_t)\mid_{\mathfrak{X}^1(R)}\in \Hom_{R}(\mathfrak{X}^1(R), R)\cong \Omega^1(R).$$ So there exists $\varpi\in \Omega^1(R)$ such that, for any $F\in \mathfrak{X}^1(R),$
$$(\Delta -\Delta_t)(F)=F(-\varpi)=(-1)\iota_{\varpi}F.$$
Since $\Delta -\Delta_t$ is a super-derivation on $\mathfrak{X}^*(R)\cong \wedge_R^*(\der(R))= \wedge_R^*(\mathfrak{X}^1(R))$ (Corollary \ref{smooth duality}), it is easy to check that $$(\Delta -\Delta_t)(P)=(-1)^p\iota_{\varpi}P$$ for any $P\in \mathfrak{X}^p(R).$ Hence $\Delta_t=\Delta-(-1)^{|\,\,|}\iota_{\varpi}$.

It follows from  the Poincar\'e duality (Remark \ref{delta'-and-iota_w}) that $\Delta_t$ corresponds to a differential $d_t=\de-(\varpi\wedge-)$ on $\Omega^*(R)$. By Lemma \ref{dt square zero}, $d_t^2=0$ if and only if $\de \varpi=0$. Then the proof is finished.
\end{proof}


\begin{cor}\label{main-cor}
Let $R$ be a  smooth Poisson algebra   with trivial canonical bundle. If its Poisson cohomology admits a BV operator which is induced from a BV operator on $\mathfrak{X}^*(R)$, then $R$ is pseudo-unimodular.
\end{cor}
\begin{proof}
Suppose $\Delta_t$ is a BV operator  on $\mathfrak{X}^*(R)$ which induces a BV operator on the Poisson cohomology. Then, by Proposition \ref{main-prop-1}, $\Delta_t=\Delta-(-1)^{|\,\,| }\iota_{\varpi}$ where $\varpi\in \Omega^1(R)$ with $\de(\varpi)=0$. By  \eqref{Detdet+detDet}, for any $x \in R,$ $$(\Delta_t\delta+\delta\Delta_t)(x)=[\Delta_t(\pi), x],$$
that is, $$\Delta_t\delta(x)=\Delta_t(\pi)( x).$$
Because $\Delta_t$  induces a differential on the Poisson cohomology and $\delta(x)$ is a 1-coboundary,  $\Delta_t(\overline{\delta(x)})=0 \in \PH^0(R)$ for any $x \in R.$
Hence $\Delta_t(\pi)=0$, that is, $\Delta(\pi)-\iota_{\varpi}(\pi)=0$. By Proposition \ref{phi_vol=Delta(pi)} (or Remark \ref{phi=DetPi}), $\phi_{\vol}=\Delta(\pi)=\iota_{\varpi}(\pi)$ and $R$ is pseudo-unimodular.
\end{proof}

\section*{Acknowledgments}
J. Luo is supported by the NSFC (Grant No. 11901396). S.-Q. Wang is supported by the NSFC (Grant No. 11771085). Q.-S. Wu is supported  by the National Key Research and Development Program of China (Grant No. 2020YFA0713200) and NSFC (Grant No. 11771085).

\end{document}